\documentclass[10pt]{amsart}
\usepackage{amsmath,amssymb,amscd,amsthm,amsfonts,amstext,amsbsy,mathrsfs,hyperref,upgreek,mathtools,stmaryrd,enumitem,bbm,graphicx,caption,subcaption,lipsum, bm} 

\usepackage[textsize=footnotesize,color=blue!40, bordercolor=white]{todonotes}
\usepackage
[a4paper, margin=1.2in]{geometry}

\definecolor{darkblue}{rgb}{0,0,0.68}
\hypersetup{colorlinks=true,
citecolor=black,
linkcolor=black,
anchorcolor=black
}

\newcommand\blfootnote[1]{%
  \begingroup
  \renewcommand\thefootnote{}\footnote{#1}%
  \addtocounter{footnote}{-1}%
  \endgroup
}

\newcommand{\RR}{\mathbb{R}}
\newcommand{\ZZ}{\mathbb{Z}}

\newcommand{\NN}{\mathbb{N}}

\newcommand{\cl}{\mathrm{cl}}
\newcommand{\Even}{\mathrm{Even}} 
\newcommand{\Odd}{\mathrm{Odd}} 

\newcommand{\ran}[1]{{{\rm{ran}}(#1)}}
\newcommand{\dom}[1]{{{\rm{dom}}(#1)}}

\newtheorem{theorem}{Theorem}[section]
\newtheorem{lemma}[theorem]{Lemma}

\newtheorem{question}[theorem]{Question}

\newtheorem{claim}[theorem]{Claim}
\newtheorem*{claim*}{Claim}
\newtheorem{subclaim}[theorem]{Subclaim}
\newtheorem*{subclaim*}{Subclaim}
\newtheorem{case}{Case}
\newtheorem*{case*}{Case}
\newtheorem{subcase}{Subcase}
\newtheorem*{subcase*}{Subcase}
\newtheorem{subsubcase}{Subsubcase}
\newtheorem*{subsubcase*}{Subsubcase}

\theoremstyle{definition}
\newtheorem{definition}[theorem]{Definition}

\theoremstyle{remark}
\newtheorem{remark}[theorem]{Remark}

\newenvironment{enumerate-(a)}{\begin{enumerate}[label={\upshape (\alph*)}, leftmargin=2pc]}{\end{enumerate}}

\newenvironment{enumerate-(a)-r}{\begin{enumerate}[label={\upshape (\alph*)}, leftmargin=2pc,resume]}{\end{enumerate}}

\newenvironment{enumerate-(A)}{\begin{enumerate}[label={\upshape (\Alph*)}, leftmargin=2pc]}{\end{enumerate}}

\newenvironment{enumerate-(A)-r}{\begin{enumerate}[label={\upshape (\Alph*)}, leftmargin=2pc,resume]}{\end{enumerate}}

\newenvironment{enumerate-(i)}{\begin{enumerate}[label={\upshape (\roman*)}, leftmargin=2pc]}{\end{enumerate}}

\newenvironment{enumerate-(i)-r}{\begin{enumerate}[label={\upshape (\roman*)}, leftmargin=2pc,resume]}{\end{enumerate}}

\newenvironment{enumerate-(I)}{\begin{enumerate}[label={\upshape (\Roman*)}, leftmargin=2pc]}{\end{enumerate}}

\newenvironment{enumerate-(I)-r}{\begin{enumerate}[label={\upshape (\Roman*)}, leftmargin=2pc,resume]}{\end{enumerate}}

\newenvironment{enumerate-(1)}{\begin{enumerate}[label={\upshape (\arabic*)}, leftmargin=2pc]}{\end{enumerate}}

\newenvironment{enumerate-(1)-r}{\begin{enumerate}[label={\upshape (\arabic*)}, leftmargin=2pc,resume]}{\end{enumerate}}

\begin{document}

\makeatletter

\makeatother



\title{Continuous reducibility and dimension}

\author{Philipp Schlicht}

\address{Mathematisches Institut, 
              Universit\"at of Bonn\\ 
              Endenicher Allee 60, 
              53115 Bonn, Germany} 
\email{schlicht@math.uni-bonn.de}




\title{Continuous reducibility and dimension of metric spaces
}\blfootnote{This paper supersedes an earlier preprint, where the main result of this paper was proved for Polish spaces, with an unrelated proof.}

\date{\today}


\maketitle

\begin{abstract} 
If $(X,d)$ is a Polish metric space of dimension $0$, then by Wadge's lemma, no more than two Borel subsets of $X$ are incomparable with respect to continuous reducibility. In contrast, our main result shows that for any metric space $(X,d)$ of positive dimension, 
there are uncountably many Borel subsets of $(X,d)$ that are pairwise incomparable with respect to continuous reducibility. 

In general, the reducibility that is given by the collection of continuous functions on a topological space $(X,\tau)$ is called the \emph{Wadge quasi-order} for $(X,\tau)$. As an application of the main result, we show that this quasi-order, restricted to the Borel subsets of a Polish space $(X,\tau)$, is a \emph{well-quasiorder (wqo)} if and only if $(X,\tau)$ has dimension $0$. 

Moreover, we give further examples of applications of the construction of graph colorings that is used in the proofs. 
\end{abstract}

\tableofcontents 

\section{Introduction}


The \emph{Wadge quasi-order} on the subsets of the Baire space 
of all functions $f\colon\NN\rightarrow \NN$ 
is an important notion that is used to fit definable sets into a hierarchy of complexity.  Its importance comes from the fact that it defines the finest known hierarchy on various classes of definable subsets of the Baire space. For instance, it refines the difference hierarchy, the Borel hierarchy and the projective hierarchy. 

The structure of the Wadge quasi-order on the Baire space and its closed subsets has therefore been an object of intensive research (see e.g. \cite{MR2906999, MR730585, MR2279651}). 
Moreover, many structural results could even be extended to many other classes of functions (see e.g. \cite{MR1992531, MR2499419, MR2605897, MR2601013, 
MR3164747, MR3308051}). 

In this paper, we study the Wadge quasi-order on the class of Borel subsets of arbitrary metric spaces. 
By a Borel subset of a topological space, we mean an element of the least $\sigma$-algebra that contains the open sets. 
We define the Wadge quasi-order for arbitrary topological spaces as follows.

\begin{definition} 
Suppose that $(X,\tau)$ is a topological space and $A$, $B$ are subsets of $X$. 
\begin{enumerate-(a)} 
\item 
$A$ is \emph{reducible} to $B$ ($A\leq_{(X,\tau)} B$) if $A=f^{-1}[B]$ for some continuous map $f\colon X\rightarrow X$. 
\item 
$A$, $B$ are \emph{equivalent} ($A\sim_{(X,\tau)} B$)
if $A\leq_{(X,\tau)} B$ and $B\leq_{(X,\tau)} A$. 
\item 
$A$, $B$ are \emph{comparable} 
if $A\leq_{(X,\tau)} B$ or $B\leq_{(X,\tau)} A$, and otherwise 
\emph{incomparable}. 
\end{enumerate-(a)}
\end{definition}

The following principle states a further important property of these quasi-orders. 

\begin{definition} 
Assuming that $(X,\tau)$ is a topological space, the \emph{semi-linear ordering principle} $\mathsf{SLO}_{(X,\tau)}$ states that for all Borel subsets $A$, $B$ of $X$, $A$ is reducible to $B$, or $B$ is reducible to $X\setminus A$. 
\end{definition} 

It is easy to see that $\mathsf{SLO}_{(X,\tau)}$ implies that no more than two Borel subsets of $X$ can be incomparable, 
Moreover, Woodin observed that $\mathsf{SLO}_{\RR}$ fails (see \cite[Remark 9.26]{Woodin_2010}, \cite[Example 3]{Andretta_SLO}) and thus it is natural to ask the following question. 

\begin{question} \label{main question} 
Under which conditions on a Polish space $(X,\tau)$ does $\mathsf{SLO}_{(X,\tau)}$ fail? 
\end{question} 

The principle 
$\mathsf{SLO}$ holds for any Polish space that is homeomorphic to a closed subset of the Baire space, since the proof of Wadge's lemma (see e.g. \cite[Section 2.3]{Andretta_SLO}) works for these spaces. 
Moreover, these Polish space are known to be exactly the ones with dimension $0$, which is defined as follows. 

\begin{definition} 
A topological space $(X,\tau)$ has \emph{dimension} $0$ if 
for every $x$ in $X$ and every open set $U$ subset of $X$ containing $x$, there is a subset of $U$ containing $x$ that is both open and closed. 
Moreover, the space has \emph{positive dimension} if it does not have dimension $0$. 
\end{definition} 

This is usually denoted by having \emph{small inductive dimension $0$}, but note that for separable metric spaces, it is equivalent to the condition that other standard notions of dimension, such as the large inductive dimension $0$ or the Lebesgue covering dimension, are equal to $0$ by \cite[Theorem 7.3.2]{MR1039321}. 
However, not every totally disconnected Polish space has dimension $0$, 
since for instance the complete Erd\"os space \cite{MR2488452} has the former property, but not the latter. 

The following is the main result of this paper, which is proved in Theorem \ref{main theorem again} below.  

\begin{theorem} \label{main theorem} 
For any metric space $(X,d)$ of positive dimension, there are uncountably many Borel subsets of $(X,d)$ that are pairwise incomparable with respect to continuous reducibility. 
\end{theorem} 



%


As an application of the main result, 
we will characterize the Wadge order on Polish spaces by the following notion, which is important in the theory of quasi-orders (see e.g. \cite{MR3268712}). 

\begin{definition} 
A \emph{well-quasiorder (wqo)} is a quasi-order with the property that there is no infinite strictly decreasing sequence and no infinite set of pairwise incomparable elements. 
\end{definition} 


The next characterization is proved in Theorem \ref{characterization result again} below and answers Question \ref{main question}.

\begin{theorem} \label{characterization result} 
Suppose that $(X,d)$ is a Polish metric space. 
Then the following conditions are equivalent. 
\begin{enumerate-(1)} 
\item[(a)] 
$X$ has dimension $0$. 
\item[(b)] 
$\mathsf{SLO}_{(X,\tau)}$ holds. 
\item[(c)] 
The Wadge order on the Borel subsets of $X$ is a well-quasiorder. 
\item[(d)] 
There are at most two pairwise incomparable Borel subsets of $X$. 
\item[(e)] 
There are at most countably many pairwise incomparable Borel subsets of $X$. 
\end{enumerate-(1)} 
\end{theorem} 

It is worthwhile to mention that Pequignot defines an alternative quasi-order on the subsets of an arbitrary Polish space \cite{MR3372615}. Moreover, his notion is more natural in the sense that it always satisfes the $\mathsf{SLO}$ principle. 

%

We will further prove the following variant of the main result in Theorem \ref{incomparable non-definable sets again} below. 

\begin{theorem} \label{incomparable non-definable sets} 
Suppose that $(X,d)$ is a locally compact metric space of positive dimension. 
Then there is a (definable) injective map that takes sets of reals to subsets of $X$ in such a way that these subsets 
are pairwise incomparable with respect to continuous reducibility. 
\end{theorem}

 
%
%
%

This paper has the following structure. 
We will prove Theorem \ref{main theorem} and Theorem \ref{characterization result} in Section  \ref{section: incomparable Borel sets}, but postpone the proof of some auxiliary results to Section \ref{section: auxiliary results}. 
Moreover, Section \ref{section: incomparable sets of arbitrary complexity} contains 
the proof of Theorem \ref{incomparable non-definable sets}, and some further remarks on the proofs can be found in Section \ref{section: further remarks}. 

We would like to thank the referee for various useful suggestions to make the paper more readable.

\section{Incomparable Borel sets} \label{section: incomparable Borel sets}

In this section, we will prove Theorem \ref{main theorem}, except for some technical steps that are postponed to the next sections.

We always let the letters $i$ to $n$ and $s$ to $w$ denote natural numbers. 
We assume that $(X,d)$ is a metric space of positive dimension and let $\alpha$, $\beta$ always denote elements of $X$. 
Moreover, for any $r> 0$ we will denote the open ball with radius $r$ around $\alpha$ by 
\[B_r(\alpha)=\{x\in X\mid d(\alpha,x)<r\}.\] 
If $\alpha$ is an element and $U$ is an open subset of $X$, then $\alpha$ is called \emph{$U$-positive} if $\alpha\in U$ and no subset $V$ of $U$ that contains $\alpha$ is open and closed in $X$. 
We further call an element $\alpha$ of $X$ \emph{positive} if it is $U$-positive for some open set $U$. 
Now $X$ has at least one positive element by the definition of positive dimension.

For each $\alpha\in X$, we fix some $r^\alpha>0$ such that $\alpha$ is $B_{r^\alpha}(\alpha)$-positive if $\alpha$ is positive and $r^\alpha$ is arbitrary otherwise. 
We further fix a strictly increasing sequence ${\bf r}^\alpha=\langle r_t^\alpha\mid t\geq 0\rangle$ with $r_0^\alpha=0$ and $\sup_{t\geq 0}r_t^\alpha=r^{\alpha}$. Let $X^{\alpha}=B_{r^{\alpha}}(\alpha)$ and $X^{\alpha}_{<t}=B_{r_t^{\alpha}}(\alpha)$ for any $t\geq 0$.







\begin{definition} \label{definition of C_A}
If $A\subseteq \mathbb{R}_{\geq 0}=\{x\in \mathbb{R}\mid x\geq 0\}$, let 
\[B^{\alpha }_A=\{x\in X^{\alpha}\mid d(\alpha,x)\in A\}\] 
if $\inf (A)>0$ and 
\[B^{\alpha }_A=\{x\in X^{\alpha}\mid d(\alpha,x)\in A\}\cup\{\alpha\}\] 
otherwise. 
If moreover $i,j\in\NN$ with $i<j$, we let 
\[C^\alpha_{(i,j)}=B^{\alpha }_{(r^\alpha_i,r^\alpha_j)}\] 
and define this in an analogous way for half-open and closed intervals from $i$ to $j$. 
\end{definition} 



We always let ${\bf m}=\langle m_i\mid i\geq 0\rangle$ and ${\bf n}=\langle n_i\mid i\geq 0\rangle$ denote strictly increasing sequences of natural numbers beginning with $0$. Moreover, we will frequently use the following notation. 

\begin{definition} \label{definition: even and odd blocks} 
Suppose that $\bf n$ is as above. 
Letting $[n_i,n_{i+1})$ denote the interval in $\mathbb{N}$, we define the following sets of natural numbers. 
\begin{enumerate-(a)} 
\item 
$\Even_{\bf n}
=\bigcup_{i\in \mathbb{N}\text{ is even}} [n_i, n_{i+1})$, 
\item 
$\Odd_{\bf n}=\bigcup_{i\in\mathbb{N}\text{ is odd}} [n_i, n_{i+1})$. 
\end{enumerate-(a)} 
\end{definition} 

In the next definition, we define sets $D_{\bf n}^\alpha$ for $\alpha$ and $\bf n$ as above. These sets will later be shown to be incomparable under appropriate assumptions.

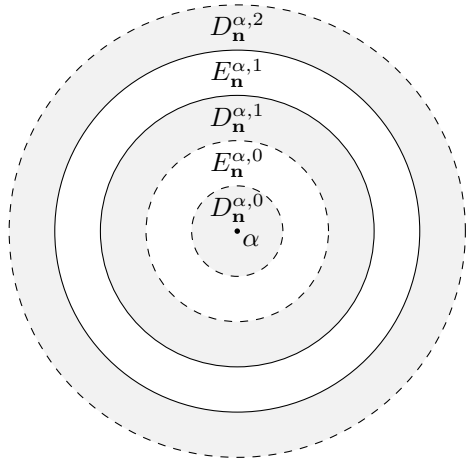
\begin{figure}[h]
\begin{center} 
\begin{tikzpicture}[scale=0.6] 

\definecolor{light}{gray}{0.95} 

\fill[fill=light] (0,0) circle (5cm); 
\draw[dashed] (0,0) circle (5cm); 

\fill[fill=white] (0,0) circle (4cm); 
\draw (0,0) circle (4cm); 
\fill[fill=light] (0,0) circle (3cm); 
\draw (0,0) circle (3cm); 
\fill[fill=white] (0,0) circle (2cm); 
\draw[dashed] (0,0) circle (2cm); 
\fill[fill=light] (0,0) circle (1cm); 
\draw[dashed] (0,0) circle (1cm); 
\fill (0,0) circle (0.06cm); 


\node at (0.3,-0.2) {$\alpha$}; 


\node at (0,0.5) {$D_{\bf n}^{\alpha,0}$};
\node at (0,1.5) {$E_{\bf n}^{\alpha,0}$};
\node at (0,2.5) {$D_{\bf n}^{\alpha,1}$};
\node at (0,3.5) {$E_{\bf n}^{\alpha,1}$};
\node at (0,4.5) {$D_{\bf n}^{\alpha,2}$};

\end{tikzpicture}
\end{center} 
\caption{A diagram of 
$D_{\bf n}^\alpha$ (shaded), assuming that $n_0=0$, $n_1=1$ and $n_2=2$. 
Solid lines consist of elements of $D_{\bf n}^\alpha$ and dotted lines of elements of $E_{\bf n}^\alpha$.  
} \label{figure concentric circles}
\end{figure}

\begin{definition} \label{definition: blocks indexed by sequences} 
Suppose that $\bf n$ is as above. 
\begin{enumerate-(a)} 
\item 
\begin{enumerate-(i)} 
\item 
Let $s$ denote the string of symbols $2j,2j+1$. 

$D_{{\bf n},j}^{\alpha}=
\begin{cases} 
C^\alpha_{[s)} & \text{if } j\in \Even_{\bf n}, \\ 
C^\alpha_{(s]} & \text{if } j\in \Odd_{\bf n}, 
\end{cases}$ 
\item 
$D_{\bf n}^\alpha=\bigcup_{j\geq 0} D_{{\bf n}, j}^{\alpha}$, 
\end{enumerate-(i)} 
\item 
\begin{enumerate-(i)} 
\item 
Let $s$ denote the string of symbols $2j+1,2j+2$ and let $u$ denote the ordered pair $(j,j+1)$. 

$E_{{\bf n},j}^{\alpha}=
\begin{cases} 
C_{[s)}^\alpha & \text{if } u\in \Even_{\bf n}\times \Even_{\bf n}, \\ 
C_{[s]}^\alpha & \text{if } u\in \Even_{\bf n}\times \Odd_{\bf n}, \\ 
C_{(s)}^\alpha & \text{if } u\in \Odd_{\bf n}\times \Even_{\bf n}, \\ 
C_{(s]}^\alpha & \text{if } u\in \Odd_{\bf n}\times \Odd_{\bf n}, 
\end{cases}$
\item 
$E_{{\bf n},j}^\alpha=\bigcup_{j\geq 0} E_{{\bf n},j}^{\alpha}$.
\end{enumerate-(i)} 
\end{enumerate-(a)} 
\end{definition}

The sets $E_{{\bf n},j}^{\alpha}$ 
are chosen to partition the complement of $D_{\bf n}^\alpha$ in $X^{\alpha}$, and thus the sets $D_{\bf n}^\alpha$, $E_{\bf n}^\alpha$ partition $X^{\alpha}$, 
as illustrated in Figure \ref{figure concentric circles}.

The idea for the proof of Theorem \ref{main theorem} in the remainder of this section is as follows. We will associate certain graphs to the sets $D^\alpha_{\bf n}$ for various sequences $\bf n$ as above. Then, we will show that for all positive $\alpha$, $\beta$, the existence of a continuous reduction $F\colon X\rightarrow X$ of $D^\alpha_{\bf m}$ to $D^\beta_{\bf n}$ 
implies the existence of certain maps associated to these graphs, and moreover, that such maps cannot exist if $\bf m$, $\bf n$ are sufficiently different. 

The combinatorics in the following proofs reflect the 
fact that it does not follow from the existence of such a function $F$ that $\bf m$ and $\bf n$ are equal, as can be seen from the example in Figure \ref{figure curve through concentric circles}. 

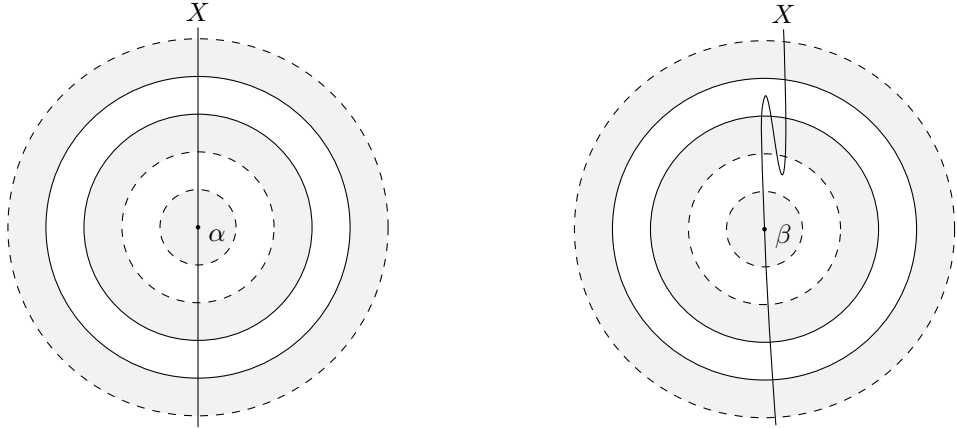
\begin{figure}[h] 
\begin{subfigure}{{.5\textwidth}}
\centering 
\begin{tikzpicture}[scale=0.5] 
\definecolor{light}{gray}{0.95} 

\fill[fill=light] (0,0) circle (5cm); 
\draw[dashed] (0,0) circle (5cm); 
\fill[fill=white] (0,0) circle (4cm); 
\draw (0,0) circle (4cm); 
\fill[fill=light] (0,0) circle (3cm); 
\draw (0,0) circle (3cm); 
\fill[fill=white] (0,0) circle (2cm); 
\draw[dashed] (0,0) circle (2cm); 
\fill[fill=light] (0,0) circle (1cm); 
\draw[dashed] (0,0) circle (1cm); 
\fill (0,0) circle (0.06cm); 

\node at (0.5,-0.2) {$\alpha$};
\node at (0,5.7) {$X$};

\draw [ xshift=0cm] plot [smooth, tension=1] coordinates {(0,-5.3) (0,0) (0,5.3)};

\end{tikzpicture}
\end{subfigure}%
\begin{subfigure}{{.5\textwidth}}
\centering 
\begin{tikzpicture}[scale=0.5] 

\definecolor{light}{gray}{0.95} 

\fill[fill=light] (0,0) circle (5cm); 
\draw[dashed] (0,0) circle (5cm); 
\fill[fill=white] (0,0) circle (4cm); 
\draw (0,0) circle (4cm); 
\fill[fill=light] (0,0) circle (3cm); 
\draw (0,0) circle (3cm); 
\fill[fill=white] (0,0) circle (2cm); 
\draw[dashed] (0,0) circle (2cm); 
\fill[fill=light] (0,0) circle (1cm); 
\draw[dashed] (0,0) circle (1cm); 
\fill (0,0) circle (0.06cm); 

\node at (0.5,-0.2) {$\beta$};
\node at (0.5,5.7) {$X$};

\draw [ xshift=0cm] plot [smooth, tension=1] coordinates {(0.3,-5.2) (0,0) (0,3.5) (0.5,1.5) (0.5,5.3)};

\end{tikzpicture} 
\end{subfigure}

\caption{Diagrams of $D_{\bf n}^\alpha$, $D_{\bf n}^\beta$ (shaded) for a one-dimensional subspace $X$ of $\RR^2$ and positive elements $\alpha$, $\beta$ of $X$. 
} \label{figure curve through concentric circles}
\end{figure} 


In (undirected) graphs, we will identify the edges, i.e. $2$-element sets of vertices, with ordered pairs. 

\begin{definition} \label{definition colored graph} 
A \emph{colored graph} consists of a graph $G$ with vertex set $V$ and edge set $E$ that satisfy the following conditions, together with a coloring $c_G$ in the colors $0$ and $1$ that is defined both on the vertices and the edges. 
\begin{enumerate-(a)} 
\item 
$V$ is a (finite or infinite) interval in $\ZZ$ of size at least $2$. 
\item 
$E$ consists of the pairs of successive vertices in $V$. 
\end{enumerate-(a)} 
\end{definition} 

We will write $c_G(i,i+1)$ for the color of the edge $(i,i+1)$ and thereby omit the additional brackets. 

\begin{definition} 
Suppose that $G$ and $H$ are colored graphs. 
\begin{enumerate-(a)} 
\item 
A \emph{reduction $f$ from $G$ to $H$} is a function defined on both the vertices and the edges of $G$ that satisfies the following conditions. 
The vertices of $G$ are mapped to vertices of $H$, the edges of $G$ are mapped to edges of $H$ and the map preserves colors in both cases. 
Moreover, if $(i,i+1)$ is an edge in $G$, then $f(i)$ and $f(i+1)$ are end points of the edge $f(i,i+1)$ in $H$. 
We will write $\dom{f}$ and $\ran{f}$ for the sets of vertices in the domain and range of $f$, respectively. 
\item 
An \emph{unfolding of $G$, $H$} is a pair $\xi=(f,g)$ such that for some finite colored graph $I$, $f$ is a reduction from $I$ to $G$ and $g$ is a reduction from $I$ to $H$. We will write 
$\dom{\xi}$
for the vertex set of $I$ and let $c_\xi=c_I$. 
\end{enumerate-(a)} 
\end{definition} 

It follows immediately from the definition of reductions that the image of a reduction $f$ is an interval in $\ZZ$, and for all $i$ with $i,i+1\in\dom{f}$ and $f(i)\neq f(i+1)$, $f(i,i+1)=(f(i),f(i+1))$. 

The followings colored graphs carry information about the sets $D_{\bf n}^\alpha$ and $E_{\bf n}^\alpha$, which were defined above. Moreover, the maps between graphs defined below carry information about possible continuous reductions from $D_{\bf m}^\alpha$ to $D_{\bf n}^\beta$ for positive elements $\alpha$, $\beta$ of $X$ and sequences $\bf m$, $\bf n$ as above.  

\begin{definition} \label{definition of Gn}
Suppose that ${\bf n}$ is as above. 
Let $G_{\bf n}$ denote the colored graph with vertex set $\NN$ and the coloring $c_{\bf n}$ defined on the vertices by 

\begin{equation*}
    c_{\bf n}(2j) =
    \begin{cases*}
      1 & if $j\in \Even_{\bf n}$ \\
      0 & if $j\in \Odd_{\bf n}$ \\       
    \end{cases*}
\end{equation*} 

\begin{equation*}
    c_{\bf n}(2j+1) =
    \begin{cases*}
      0 & if  $j\in \Even_{\bf n}$ \\
      1 & if  $j\in \Odd_{\bf n}$ \\       
    \end{cases*}
\end{equation*} 
for $j\geq 0$ and on the edges by 

\begin{equation*}
    c_{\bf n}(j,j+1) =
    \begin{cases*}
      1 & if \text{$j$ is even} \\
      0 & if \text{$j$ is odd}  \\       
    \end{cases*}
\end{equation*} 
for $j\geq 0$. 
\end{definition}

When ${\bf m}$, ${\bf n}$ are as above and $\xi$ is an  unfolding, we will always mean that it is an unfolding of $G_{\bf m}$, $G_{\bf n}$.

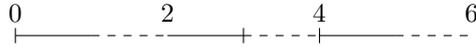
\begin{figure}[h]
\begin{tikzpicture} 
\draw (0,0) -- (1,0);
\node at (0,0.3) {$0$};
\draw (0,-0.1) -- (0,0.1);
\draw[dashed] (1,0) -- (2,0);
\draw (2,0) -- (3,0);
\node at (2,0.3) {$2$};
\draw (3,-0.1) -- (3,0.1);
\draw[dashed] (3,0) -- (4,0);
\draw (4,0) -- (5,0);
\node at (4,0.3) {$4$};
\draw (4,-0.1) -- (4,0.1);
\draw[dashed] (5,0) -- (6,0);
\node at (6,0.3) {$6$};

\end{tikzpicture} 

\caption{A diagram of the graph $G_{\bf n}$ corresponding to Figures \ref{figure concentric circles} and \ref{figure curve through concentric circles}.
The solid lines denote edges with color $1$ and the dotted lines denote edges with color $0$. Moreover, the marked vertices have color $1$ and the other ones have color $0$. }
\end{figure} 

\begin{definition} 
Suppose that $\bf m$, $\bf n$ are as above and $\xi=(f,g)$ is an unfolding of $G_{\bf m}$, $G_{\bf n}$, where $f$ is a reduction from $I$ to $G_{\bf m}$ and $g$ is a reduction from $I$ to $G_{\bf n}$. 
We define the relation $\sim_\xi$ on pairs of vertices and pairs of edges as follows. 
\begin{enumerate-(a)} 
\item 
If $k$ is a vertex in $G_{\bf m}$ and $l$ is a vertex in $G_{\bf n}$, 
let $k\sim_{\xi}l$ if $f(j)=k$ and $g(j)=l$ for some vertex $j$ in $I$. 
\item 
If $v$ is an edge in $G_{\bf m}$ and $w$ is an edge in $G_{\bf n}$, let $v \sim_{\xi} w$ if $f(u)=v$ and $g(u)=w$ for some edge $u$ in $I$. 
\end{enumerate-(a)} 
\end{definition} 


We will from now on always assume that we are in the following situation. 
We assume that $\alpha$ is a positive and $\beta$ is an arbitrary element of $X$. 
The sequences $\bf m$, $\bf n$ are strictly increasing sequences in $\NN$ beginning with $0$ as above and $F\colon X\rightarrow X$ is a continuous function with $D^\alpha_{\bf m}\cap Y^{\alpha,\beta}=F^{-1}[D^\beta_{\bf n}]\cap Y^{\alpha,\beta}$, where $Y^{\alpha,\beta}=X^{\alpha}\cap F^{-1}[X^{\beta}]$.


\begin{definition} \label{definition compatible} 
Suppose that $\bf m$, $\bf n$, $\alpha$, $\beta$, $F$ are as above. 
Moreover, suppose that $x\in X^{\alpha }$ and $\xi=(f,g)$ is an unfolding of $G_{\bf m}$, $G_{\bf n}$. 
We say that \emph{$x$ and $\xi$ are compatible} with respect to ${\bf m}$, ${\bf n}$, $\alpha$, $\beta$, $F$ if one of the following conditions holds. 
\begin{enumerate-(a)} 
\item 
$(x,F(x))\in D^{\alpha}_{{\bf m},i}\times D^{\beta}_{{\bf n},j}$ for some pair $(i,j)$ with 
$$(2i,2i+1)\sim_\xi (2j,2j+1).$$ 
\item 
$(x,F(x))\in E^{\alpha}_{{\bf m},i}\times E^{\beta}_{{\bf n},j}$ for some pair $(i,j)$ with 
$$(2i+1,2i+2)\sim_\xi (2j+1,2j+2).$$ 
\end{enumerate-(a)} 
We will omit ${\bf m}$, ${\bf n}$, $\alpha$, $\beta$, $F$ if they are clear from the context. 
Moreover, we define the \emph{compatibility range} of $\xi$ with respect to these parameters 
as the set $X_{{\bf m}, {\bf n}, \xi}^{\alpha,\beta,F}$ of all $x\in X^{\alpha}$ such that $x$ and $\xi$ are compatible. 
\end{definition}

\begin{definition} \label{definition X_i} 
Suppose that $\bf m$, $\bf n$, $\alpha$, $\beta$, $F$ are as above, $i\geq 0$ and $t\geq 1$. 
\begin{enumerate-(a)} 
\item 
Let $\mathbb{U}_{{\bf m}, {\bf n},i}$ denote the set of unfoldings $\xi$ of $G_{\bf m}$, $G_{\bf n}$ with $0\sim_\xi i$ 
and $\mathbb{U}_{{\bf m}, {\bf n},i,{<}t}$ the set of $\xi=(f,g)\in \mathbb{U}_{{\bf m}, {\bf n},i}$ with $t\notin \ran{f}$. 
\item 
Let $X_{{\bf m}, {\bf n},i,<t}^{\alpha, \beta,F}=X^\alpha_{<t}\cap \bigcup_{\xi\in \mathbb{U}_{{\bf m}, {\bf n},i,<t}} X_{{\bf m}, {\bf n}, \xi}^{\alpha,\beta,F}$. 
\end{enumerate-(a)} 
\end{definition}

The key to the proof of the main theorem is given by the next lemma, whose proof is postponed to Section \ref{section proofs Xi is open and closed} due to its technical nature. 

\begin{lemma} \label{Xi is open intro} \label{Xi is closed intro} 
For all $\bf m$, $\bf n$, $\alpha$, $\beta$, $F$ as above with $\Delta{\bf m}$, $\Delta{\bf n}$ strictly increasing, $s\geq 0$ and $t\geq 1$, $X_{{\bf m}, {\bf n},s,<t}^{\alpha,\beta,F}$ is an open and relatively closed subset of $X^{\alpha}_{<t}$. 
\end{lemma} 


Note that it follows 
that $X_{{\bf m}, {\bf n},s,<t}^{\alpha,\beta,F}$ is an open subset of the full space $X$, since $X^{\alpha}$ is open. However, 
we do not require that $X_{{\bf m}, {\bf n},s,<t}^{\alpha,\beta,F}$ is a closed subset of $X$.

The next step in the proof of Theorem \ref{main theorem} is given by the following result, which is used only for $\alpha=\beta=\gamma$ there and later, in its general form, in the proof of Theorem \ref{incomparable non-definable sets}.

\begin{lemma} \label{existence of reduction implies tail equivalence} 
Suppose that $\bf m$, $\bf n$, $\alpha$, $\beta$, $F$ are as above with $\Delta \bf m$, $\Delta \bf n$ strictly increasing 
and some $\gamma\in X^\alpha_{<1}\cap Y^{\alpha,\beta}$ is $X^\alpha$-positive. 
Then 
$$\bigcup_{(f,g)\in \mathbb{U}_{{\bf m}, {\bf n},l}} \ran{f}=\NN$$ 
for some $l\geq 0$. 
\end{lemma} 
\begin{proof} 
We assumed before Definition \ref{definition compatible} that $F\colon X\rightarrow X$ is a continuous function with 
$D^\alpha_{\bf m}\cap Y^{\alpha,\beta}=F^{-1}[D^\beta_{\bf n}]\cap Y^{\alpha,\beta}$, where 
$Y^{\alpha,\beta}=X^{\alpha}\cap F^{-1}[X^{\beta}]$ as given before Definition \ref{definition compatible}. 
Now the assumptions on $\gamma$ imply that $\gamma\in D^\alpha_{\bf m}\cap Y^{\alpha,\beta}$ and hence $F(\gamma)\in D^\beta_{\bf n}$. 
Since the sets $D^{\beta}_{{\bf n},j}$ are pairwise disjoint for different $j\geq 0$, there is a unique $l\geq 0$ with $F(\gamma)\in D^{\beta}_{{\bf n},l}$. 
We will prove the next two claims for arbitrary $t\geq 1$ and write $X_{2l}=X^{\alpha,\beta,F}_{{\bf m},{\bf n},2l,<t}$.

\begin{claim*} 
$\gamma \in X_{2l}$. 
\end{claim*} 
\begin{proof} 
By the definition of $X_{2l}=X^{\alpha,\beta,F}_{{\bf m},{\bf n},2l,<t}$ in Definition \ref{definition X_i}, it is sufficient to show that there is some unfolding $\xi=(f,g)$ of $G_{\bf m}$, $G_{\bf n}$ with $0\sim_{\xi}2l$ and $t\notin \ran{f}$ such that $\gamma$ and $\xi$ are compatible with respect to ${\bf m}$, ${\bf n}$, $\alpha$, $\beta$, $F$. We consider the colored graph $I$ with vertex set $\{0,1\}$, vertex colors $c_I(0)=c_I(1)=1$ and edge color $c_I(0,1)=1$. 

Since $c_{\bf m}(0)=1$ and $c_{\bf m}(0,1)=1$ by the definition of $c_{\bf m}$, we can define a reduction $f$ from $I$ to $G_{\bf m}$ by 
$$f(0)=f(1)=0$$ 
$$f(0,1)=(0,1).$$ 
If $l\in\mathrm{Even}_{\bf n}$, then $c_{\bf n}(2l)=1$ and $c_{\bf n}(2l,2l+1)=1$ by the definition of $c_{\bf n}$. Therefore, we can define a reduction $g$ from $I$ to $G_{\bf n}$ by 
$$g(0)=g(1)=2l$$ 
$$g(0,1)=(2l,2l+1).$$ 
If $l\in \mathrm{Odd}_{\bf n}$, then we let  
$g(0)=g(1)=2l+1$ and leave the remaining values of $f$, $g$ unchanged. 

Thus $\xi=(f,g)$ is an unfolding of $G_{\bf m}$, $G_{\bf n}$ with $1\notin \ran{f}$. 
We have $\gamma\in D_{{\bf m},0}^{\alpha}$ by the definition of $D_{{\bf m},0}^{\alpha}$ and $F(\gamma)\in D_{{\bf n},l}^{\beta}$ by the choice of $l$. 
Hence $\gamma$ and $\xi$ are compatible with respect to ${\bf m}$, ${\bf n}$, $\alpha$, $\beta$, $F$ as witnessed by the fact that $(0,1)\sim_{\xi}(2l,2l+1)$. 
\end{proof} 

\begin{claim*}
For every $r$ with $d(\alpha,\gamma)\leq r<r^{\alpha}_t$, there is some $x\in X_{2l}$ 
with $d(\alpha,x)=r$. 
\end{claim*} 
\begin{proof} 
By the previous claim, we can assume that $r>d(\alpha,\gamma)$. 
Towards a contradiction, suppose that there is no $x\in X_{2l}$ with $d(\alpha,x)=r$. It follows that $X_{2l}\cap B_r(\alpha)=X_{2l}\cap \bar{B}_r(\alpha)$, where 
\[\bar{B}_r(\alpha)=\{x\in X\mid d(\alpha,x)\leq r\}\] 
denotes the closed ball of radius $r$ around $\alpha$ in $X$. 
Let 
$$U=X_{2l}\cap B_r(\alpha)=X_{2l}\cap\bar{B}_r(\alpha).$$ 
We have that $X_{2l}\cap B_r(\alpha)$ is open by Lemma \ref{Xi is open intro} and $X_{2l}\cap \bar{B}_r(\alpha)$ is relatively closed in $X^{\alpha}_{<t}$ by Lemma \ref{Xi is closed intro} and hence closed in $X$. 
Thus $U$ is a subset of $X^{\alpha}$ with $\gamma\in X^{\alpha}$ that is both open and closed in $X$. 
However, this contradicts the assumption that $\gamma$ is $X^{\alpha}$-positive. 
\end{proof}

\begin{claim*} 
$\bigcup_{(f,g)\in \mathbb{U}_{{\bf m}, {\bf n},2l}} \ran{f}=\NN.$ 
\end{claim*} 
\begin{proof} 
Suppose that $j\geq 0$ and $r$ is chosen with $r_{2j}^\alpha<r<r_{2j+1}^\alpha$ and $d(\alpha,\gamma)<r$. 
By the previous claim for $t=2j+1$, there is some $x\in X_{2l}$ with $d(\alpha,x)=r$. 
In particular, we have $x\in D_{{\bf n},j}^{\alpha}$ by the definition of $D_{{\bf n},j}^{\alpha}$.

By the definition of $X_{2l}=X^{\alpha,\beta,F}_{{\bf m},{\bf n},2l,<t}$ in Definition \ref{definition X_i}, 
there is some unfolding $\xi=(f,g)\in\mathbb{U}_{{\bf m},{\bf n},2l}$ 
with $t\notin \ran{f}$ and $x\in X_{{\bf m},{\bf n},\xi}^{\alpha,\beta,F}$. 
Thus $x$ and $\xi$ are compatible with respect to ${\bf m}$, ${\bf n}$, $\alpha$, $\beta$, $F$. 
Together with the fact that $x\in D_{{\bf n},j}^{\alpha}$ and by the definition of compatibility in Definition \ref{definition compatible}, we thus obtain that $2j\in \ran{f}$. 
Since moreover $0\in \mathrm{ran}(f)$ by the definition of $\mathbb{U}_{{\bf m},{\bf n},2l}$ and since $\ran{f}$ is an interval in $\ZZ$ by the definition of reductions, it follows that 
$\{0,\dots,2j\}\subseteq \ran{f}$. 
\end{proof} 

The last claim completes the proof of Lemma \ref{existence of reduction implies tail equivalence}. 
\end{proof}

Given the preceding lemma, the remaining step 
is to show that $\Delta{\bf m}$, $\Delta{\bf n}$ are $\mathrm{E}_{\mathrm{tail}}$-equivalent, where  $\Delta{\bf m}$, $\Delta{\bf n}$ are as above and the equivalence relation $\mathrm{E}_{\mathrm{tail}}$ is defined as follows.


\begin{definition} 
Suppose that $\bf m$, $\bf n$ are as above. 
\begin{enumerate-(a)} 
\item 
We define an equivalence relation $\mathrm{E}_{\mathrm{tail}}$ on the set of sequences $\bf n$ as above as follows. 
Two sequences $\bf m$, $\bf n$ are \emph{$\mathrm{E}_{\mathrm{tail}}$-equivalent} 
if there is some $i_0\geq 0$ and some $j\in \ZZ$ such that $m_i=n_{i+j}$ for all $i\geq i_0$. 
\item 
Let $\Delta \bf m$ denote the sequence with values $({\Delta \bf m})_i=m_{i+1}-m_i$ for all $i\geq 0$. 
\end{enumerate-(a)} 
\end{definition}

The proof of $\mathrm{E}_{\mathrm{tail}}$-equivalence uses the next result, whose proof is postponed to Section \ref{section proof of a lemma} due to its technical nature.

\begin{lemma} \label{unfoldings with large domain implies tail equivalence} 
Suppose that $\bf m$, $\bf n$ are as above 
with $\Delta \bf m$, $\Delta \bf n$ strictly increasing. Moreover, suppose that 
$$\bigcup_{(f,g)\in \mathbb{U}_{{\bf m}, {\bf n},l}} \ran{f}=\NN$$ 
for some $l\geq 0$. 
Then $\Delta \bf m$, $\Delta \bf n$ are $\mathrm{E}_{\mathrm{tail}}$-equivalent. 
\end{lemma}

We are now ready to complete the proof of the main result. 

\begin{theorem} \label{main theorem again} 
For any metric space $(X,d)$ of positive dimension, there are uncountably many Borel subsets of $(X,d)$ that are pairwise incomparable with respect to continuous reducibility. 
\end{theorem} 

\begin{proof} 
Let $A$ denote the subset of the Baire space $\NN^{\NN}$ consisting of the sequences $\bf n$ beginning with $0$ such that both $\bf n$ and $\Delta{\bf n}$ are strictly increasing. Since $A$ is a closed subset of $\NN^{\NN}$, it is itself a Polish space. Moreover, we consider the equivalence relation $E$ on $A$ defined by 
\[ ({\bf m},{\bf n})\in E\Longleftrightarrow (\Delta{\bf m},\Delta{\bf n})\in E_{\mathrm{tail}}.\] 
Then $E$ is a Borel equivalence relation on $A$ whose equivalence classes are countable. Therefore, it is easy to check and follows from standard results (see e.g. \cite[Lemma 32.2]{Jech} and \cite[Theorem 8.41]{Kechris_Classical}) that there is a perfect subset of $A$ whose elements are pairwise $E$-inequivalent. 
Let $\alpha$ be any positive element of $X$, as defined in the beginning of this section. 
The claim now follows from Lemma \ref{existence of reduction implies tail equivalence} for $\alpha=\beta=\gamma$ and Lemma \ref{unfoldings with large domain implies tail equivalence}. 
\end{proof} 

Moreover, we can now easily obtain the following application of the main result. 

\begin{theorem} \label{characterization result again} 
Suppose that $(X,d)$ is a Polish metric space. 
Then the following conditions are equivalent. 
\begin{enumerate-(1)} 
\item[(a)] 
$X$ has dimension $0$. 
\item[(b)] 
$\mathsf{SLO}_{(X,\tau)}$ holds. 
\item[(c)] 
The Wadge order on the Borel subsets of $X$ is a well-quasiorder. 
\item[(d)] 
There are at most two pairwise incomparable Borel subsets of $X$. 
\item[(e)] 
There are at most countably many pairwise incomparable Borel subsets of $X$. 
\end{enumerate-(1)} 
\end{theorem}

\begin{proof} 
It is well-known that any Polish metric space $(X,d)$ of dimension $0$ is homeomorphic to the set of branches $[T]$ of some subtree $T$ of the tree $\NN^{<\NN}$ of finite sequences of natural numbers, ordered by inclusion. Since the proof of Wadge's lemma (see e.g. \cite[Section 2.3]{Andretta_SLO}) works for these spaces, it follows from Borel determinacy that $\mathsf{SLO}_{(X,d)}$ holds.  
The remaining implications follow immediately from Theorem \ref{main theorem again}. 
\end{proof}


\section{Auxiliary results} \label{section: auxiliary results} 
In this section we complete the missing proofs of Lemma \ref{Xi is open intro} and \ref{unfoldings with large domain implies tail equivalence} from the previous section.


\subsection{Unfoldings and $\mathrm{E}_{\mathrm{tail}}$-equivalence
}
\label{section proof of a lemma} 


In this section we will prove Lemma \ref{unfoldings with large domain implies tail equivalence} and several other auxiliary results about unfoldings. 
We always let ${\bf m}$, ${\bf n}$ denote strictly increasing sequences in $\NN$ beginning with $0$. 
Whenever $\xi$, $\xi^*$ are unfoldings, we say that $\xi^*$ \emph{extends} $\xi$ if every pair or vertices or edges that appears in $\sim_\xi$ also appears in $\sim_{\xi^*}$. 

\begin{lemma} \label{extend unfolding} 
In the following, $\xi$, $\xi^*$ will denote unfoldings of $G_{\bf m}$, $G_{\bf n}$. 
Suppose that $\xi=(f,g)$, $k,l>0$ and $c_{\bf m}(k)=c_{\bf n}(l)$. 
If one of the conditions 
$$(k,k+1)\sim_{\xi}(l,l+1)$$ 
$$(k-1,k)\sim_{\xi}(l-1,l)$$ 
$$(k,k+1)\sim_{\xi}(l-1,l)$$ 
$$(k-1,k)\sim_{\xi}(l,l+1)$$ 
holds, then there is some $\xi^*=(f^*,g^*)$ that extends $\xi$ and satisfies the respective condition in the list 
$$(k-1,k)\sim_{\xi^*}(l-1,1)$$ 
$$(k,k+1)\sim_{\xi^*}(l,l+1)$$ 
$$(k-1,k)\sim_{\xi^*}(l,l+1)$$ 
$$(k,k+1)\sim_{\xi^*}(l-1,l)$$ 
and moreover $\ran{f^*}=\ran{f}\cup\{k\}$, $\ran{g^*}=\ran{g}\cup\{l\}$. 
\end{lemma} 
\begin{proof} 
Suppose that $\dom{\xi}=[u,u^*]$. We begin with the first case. 
By the assumption and by the definition of $\sim_\xi$, there is some $j\in\dom{\xi}$ with $f(j,j+1)=(k,k+1)$ and $g(j,j+1)=(l,l+1)$. 

We define the following coloring on the vertices in the interval $[u-1,u^*+1]$ and on the edges between adjacent vertices, and thereby define a colored graph $I$. 


\begin{equation*}
    c_I(i) =
    \begin{cases*}
      c_\xi(i+1) & if $i\leq j-1$ \\
      c_{\bf m}(k)=c_{\bf n}(l) & if $i\in\{j,j+1\}$ \\
      c_\xi(i-1) & if $i\geq j+2$ \\
    \end{cases*}
\end{equation*} 

\begin{equation*}
    c_I(i,i+1) =
    \begin{cases*}
      c_\xi(i+1,i+2) & if $i\leq j-1$ \\
      1-c_{\bf m}(k,k+1)=1-c_{\bf n}(l,l+1) & if $i=j$ \\
      c_\xi(i-1,i) & if $i\geq j+1$ \\
    \end{cases*}
\end{equation*} 

We further define the following maps $f^*$, $g^*$ on the vertices and edges of $I$. 

\begin{equation*}
    f^*(i) =
    \begin{cases*}
      f(i+1) & if $i\leq j-1$ \\
      k & if $i\in\{j,j+1\}$ \\
      f(i-1) & if $i\geq j+2$ \\
    \end{cases*}
\end{equation*} 


\begin{equation*}
    f^*(i,i+1) = 
    \begin{cases*}
       f(i+1,i+2) & if $i\leq j-1$ \\
      (k-1,k) & if $i=j$ \\
      f(i-1,i) & if $i\geq j+1$  \\
    \end{cases*}
\end{equation*} 


The definition of $g^*$ is obtained from the definition of $f^*$ by replacing $f$, $k$ by $g$, $l$.  
The statement that $\xi^*=(f^*,g^*)$ is an unfolding that extends $\xi$ can be checked from the definitions above. 
Moreover, the statements about the ranges of $f^*$, $g^*$ follows from the definitions above. 

We now describe the changes to the above definitions in the remaining three cases. 
Here $j$ is chosen such that $f(j,j+1)$, $g(j,j+1)$ are equal to the two given pairs that are in the relation $\sim_\xi$ by one of the conditions. 
In the definition $c_I(i,i+1) = 1-c_{\bf m}(k,k+1)=1-c_{\bf n}(l,l+1)$ for $i=j$, the arguments $(k,k+1)$, $(l,l+1)$ are replaced by the two given pairs. 
The definitions of $f^*$, $g^*$ remain the same if the first pair is $(k,k+1)$ or the second pair is $(l,l+1)$, respectively. If the first pair is $(k-1,k)$, then the definition of $f^*$ is changed to $f^*(i,i+1)=(k,k+1)$ for $i=j$, and similarly, if the second pair is $(l-1,l)$, then we let $g^*(i,i+1)=(l,l+1)$ for $i=j$. 
\end{proof}

For the statement of the next result, we fix the following notation. If $G$ is a colored graph and $V$ is a sub-interval of the vertex set of $G$, then we let $G^V$ denote the unique colored graph with vertex set $V$ that is induced by $G$. 




\begin{lemma} \label{unfolding induces reduction} 
Suppose that $\bf m$, $\bf n$ are as above, $s,t\geq 0$ and $\xi=(f,g)$ is an unfolding of $G_{\bf m}$, $G_{\bf n}$ with $$\ran{f}\subseteq [2m_s,2m_{s+1}-1]$$ $$\ran{g}\subseteq[2n_t,2n_{t+1}-1].$$ 
Then the following statements hold. 
\begin{enumerate-(1)} 
\item 
The equation 
$$(*_{i,j})\ \ \ \ \ \ 
f(i)-f(j)=(-1)^{s+t} (g(i)-g(j)) $$ 
holds for all $i,j\in\dom{\xi}$. 
\item 
We obtain a unique reduction from $G_{\bf m}^{\ran{f}}$ to $G_{\bf n}^{\ran{g}}$ by restricting $fg^{-1}$ to the vertices and edges of $G_{\bf m}^{\ran{f}}$. 
The same holds for $gf^{-1}$ when $G_{\bf m}^{\ran{f}}$ is replaced with $G_{\bf n}^{\ran{g}}$ and conversely. 
\end{enumerate-(1)} 
\end{lemma} 
\begin{proof} 
To prove the first claim, suppose that $f$ is a reduction from $I$ to $G_{\bf m}$ and $g$ is a reduction from $I$ to $G_{\bf n}$, where $I$ is a finite colored graph. 
It is sufficient to prove the claim for $i\leq j$, since this implies the claim for $i\geq j$. 
We now fix $i$ and prove the formula $(*_{i,j})$ by induction on $j\geq i$. Therefore, we assume that the formula holds for some $j\geq i$ with $j+1\in\dom{\xi}$. 

\begin{case*} 
$f(j+1)=f(j)$. 
\end{case*} 
Since $f$, $g$ are reductions, the case assumption implies that  
\[c_{\bf n}(g(j+1))=c_I(j+1)=c_{\bf m}(f(j+1))=c_{\bf m}(f(j))=c_I(j)=c_{\bf n}(g(j)).\] 
Since the colors of vertices alternate in $[2n_t,2n_{t+1}-1]$ by the definition of $G_{\bf n}$, and since $g$ is a reduction, this implies that $g(j+1)=g(j)$. Thus the formula $(*_{i,j+1})$ follows immediately from the formula $(*_{i,j})$ that is given by the induction hypothesis. 

\begin{case*} 
$f(j+1)\neq f(j)$. 
\end{case*} 
We can assume that $f(j+1)>f(j)$, since the case that $f(j+1)<f(j)$ is symmetric. Using this assumption and the fact that $f$ is a reduction, it follows that $f(j+1)=f(j)+1$. 
We can further assume that $s$, $t$ are even and that $f(j)$ is even, because the remaining cases are symmetric. Since $f$, $g$ are reductions, $s$ is even, $f(j)$ is even and by the definition of $G_{\bf m}$, we have 
\setcounter{equation}{0}
\begin{equation} \label{equation 1} 
c_{\bf n}(g(j))=c_I(j)=c_{\bf m}(f(j))=1 
\end{equation} 
\begin{equation} \label{equation 3} 
c_{\bf n}(g(j),g(j+1))=c_I(j,j+1)=c_{\bf m}(f(j),f(j+1))=c_{\bf m}(f(j),f(j)+1)=1. 
\end{equation} 

Since $t$ is even and by the definition of $G_{\bf n}$, equation (\ref{equation 1}) 
implies that $g(j)$ is even and $g(j+1)$ is odd. 
Finally, since we argued that $g(j)$ is even, since $t$ is even and by the definition of $G_{\bf n}$, 
equation (\ref{equation 3}) implies that $g(j+1)=g(j)+1$. 
Since $s$, $t$ are assumed to be even, $(-1)^{s+t}=1$. Thus the formula $(*_{i,j+1})$ follows immediately from the statements $f(j+1)=f(j)+1$ and $g(j+1)=g(j)+1$. 

We now indicate the modifications to the previous proof in the situation that $f(j+1)=f(j)$ for the remaining cases. 
In all cases with $f(j+1)<f(j)$ below, we will write $f(j)-1$ instead of $f(j)+1$ in equation (\ref{equation 3}).

We first assume that $s$, $t$ have the same parity. It follows that $f(j)$, $g(j)$ have the same parity by equation (\ref{equation 1}). 
If $f(j+1)>f(j)$ and $f(j)$ is even, then equation (\ref{equation 3}) takes the value $1$. Since it follows from the case assumption that $g(j)$ is also even, this implies that $g(j+1)=g(j)+1$. 
Similarly, if $f(j+1)<f(j)$ and $f(j)$ is odd, then equation (\ref{equation 3}) takes the same value, and it follows from this and the fact that $g(j)$ is odd that $g(j+1)=g(j)-1$. 

If $f(j+1)>f(j)$ and $f(j)$ is odd, then equation (\ref{equation 3}) takes the value $0$. Since it follows from the case assumption that $g(j)$ is also odd, this implies that $g(j+1)=g(j)+1$. 
Similarly, if $f(j+1)<f(j)$ and $f(j)$ is even, then equation (\ref{equation 3}) takes the same value, $g(j)$ is even and $g(j+1)=g(j)-1$.

We now assume that $s$, $t$ have different parity and hence $f(j)$, $g(j)$ have different parity by equation (\ref{equation 1}). 
If $f(j+1)>f(j)$ and $f(i)$ is even, then (\ref{equation 3}) has the value $1$. Since $g(j)$ is then odd, this shows that $g(j+1)=g(j)-1$. 
If $f(j+1)<f(j)$ and $f(j)$ is odd, then the (\ref{equation 3}) has the same value, and it follows from this and the fact that $g(j)$ is even that $g(j+1)=g(j)+1$. 

If $f(j+1)>f(j)$ and $f(j)$ is odd, then (\ref{equation 3}) has the value $0$. Since $g(j)$ is then even, this implies that $g(j+1)=g(j)-1$. 
Finally, if $f(j+1)<f(j)$ and $f(j)$ is even, then (\ref{equation 3}) has the same value, $g(j)$ is odd and $g(j+1)=g(j)+1$.


We now prove the second claim. 
Note that it is immediate from the first claim that the restriction of $fg^{-1}$ to vertices is a well-defined map on $\ran{g}$. 
Moreover, we can assume that $s+t$ is even, since the case that $s+t$ is odd is symmetric. Then $s$, $t$ have the same parity. If $i_0\in \ran{g}$ and $g(j_0)=i_0$, then by the first claim, $fg^{-1}(i_0+i)=f(j_0)+i$ for all $i$ with $i_0+i\in \ran{g}$. 
It now follows from the fact that $f$, $g$ are reductions that 
$$fg^{-1}(i_0+i,i_0+i+1)=(f(j_0)+i,f(j_0)+i+1)$$ 
for all $i$ with $i_0+i,i_0+i+1\in \ran{g}$, and thus the described restriction of $fg^{-1}$ is a reduction. 

In the symmetric case that $i+j$ is odd, we replace $f(j_0)+i$ with $f(j_0)-i$ in the previous statements. 
\end{proof} 

By keeping the range of $f$ restricted to a single interval but allowing $g$ to range over several intervals, we obtain the following variant of the previous lemma. 


\begin{lemma} \label{sum formula} 
Suppose that $\bf m$, $\bf n$ are as above and $\xi=(f,g)$ is an unfolding of $G_{\bf m}$, $G_{\bf n}$. Moreover, suppose that $a\in \dom{\xi}$, $s,t\geq 0$ and $w\geq t$ are such that 
$$\ran{f}\subseteq [2m_s,2m_{s+1}-1]$$ 
$$\ran{g}\subseteq[g(a),2n_w-1]$$ 
$$g(a)\in [2n_t,2n_{t+1}-1].$$ 
Let
\begin{equation*}
    b_u =
    \begin{cases*}
      g(a) & if $u=t$ \\
      2n_u & if $t<u\leq w$, \\
    \end{cases*}
\end{equation*} 

\[p_u=(b_{u+1}-b_u)-1\]
for $t\leq u< w$ and 
\[\Sigma_v=\sum_{t\leq u<v} (-1)^u p_u.\] 
for $t\leq v\leq w$. 

\begin{enumerate-(1)} 
\item 
If $i\in\dom{\xi}$, let $v$, $j$ be unique with $t\leq v<w$, $j\leq p_v$ and $g(i)=b_v+j<b_{v+1}$. Then 
\[
(*_i) \ \ \ \ \ \ 
f(i)=f(a)+(-1)^s (\Sigma_v+(-1)^v j)\] 
\item \label{distance statement}
Assume that $\Delta \bf m$, $\Delta \bf n$ are strictly increasing and $t\leq v<v+1=w$. 
If $\ran{g}\subseteq [g(a),2n_w-1]$, then 
$|f(i)-f(i^*)|\leq p_v$ 
for all $i,i^*\in \dom{\xi}$. 
If moreover $\ran{g}\subseteq [g(a),2n_w-1)$, then we obtain the strict inequality $|f(i)-f(i^*)|< p_v$. 
\end{enumerate-(1)} 
\end{lemma} 
\begin{proof} 
To prove the first claim, suppose that $f$ is a reduction from $I$ to $G_{\bf m}$ and $g$ is a reduction from $I$ to $G_{\bf n}$, where $I$ is a finite colored graph. 

We prove the formula $(*_i)$ by induction on $i\geq a$. The proof for $i\leq a$ is symmetric (by replacing $i+1$ with $i-1$ in the equations below).  
Let $a^*=\max(\dom{\xi})$. Suppose that the formula $(*_i)$ holds for some $i$ with $a\leq i<a^*$. 

We partition $[a,a^*]$ into maximal subintervals 
$[a_0,a_0^*],\dots,[a_l,a_l^*]$ 
such that for all $j<l$, there is some $u$ with $t\leq u<w$ and 
$$\ran{g{\upharpoonright}[a_j,a_{j+1}]}\subseteq[b_u,b_{u+1}-1].$$ 
Moreover, we can assume that the intervals are ordered such that 
$a_j<a_{j+1}$ for all $j<l$. 

\begin{case*} 
$i=a_k^*$ for some $k<l$. 
\end{case*} 

By the choice of the subintervals and since $i<a$, we have $i+1=a_k^*+1=a_{k+1}$ and there is some $v$ with $t\leq v+1<w$ and 
\setcounter{equation}{0}
\begin{equation} \label{equation a1} 
\{g(i),g(i+1)\}=\{g(a_k^*),g(a_{k+1})\}=\{b_{v+1}-1,b_{v+1}\}.
\end{equation} 

Since $c_{\bf n}(b_{v+1}-1)=c_{\bf n}(b_{v+1})$ by the definition of $G_{\bf n}$, it follows from equation (\ref{equation a1}) and the fact that $f$, $g$ are reductions that 
\begin{equation} \label{equation a2} 
c_{\bf m}(f(i))=c_I(i)=c_{\bf n}(g(i))=c_{\bf n}(g(i+1))=c_I(i+1)=c_{\bf m}(f(i+1)). 
\end{equation} 

Since the colors of vertices in $[2m_s,2m_{s+1}-1]$ alternate by the definition of $G_{\bf m}$ and since $\ran{f}$ is contained in this interval by assumption, it follows from equation (\ref{equation a2}) that $f(i)=f(i+1)$. 

Since $b_{v+1}-1=b_v+p_v$ and $\Sigma_v+(-1)^v p_v=\Sigma_{v+1}$, it follows from the equation (\ref{equation a1}) that the formulas $(*_i)$ and $(*_{i+1})$ yield the same value. 
Since we argued that $f(i)=f(i+1)$ and the formula $(*_{i})$ holds by the induction hypothesis, this implies that $(*_{i+1})$ holds. 

\begin{case*} 
$i\in [a_k,a_k^*)$ for some $k\leq l$. 
\end{case*} 

Suppose that $t\leq u<w$ and $\ran{g{\upharpoonright}[a_k,a_k^*]}\subseteq [b_u,b_{u+1}-1]$. 
Moreover, let $j,j^*\leq p_u$ be the unique numbers with $g(i)=b_u+j<b_{u+1}$ and 
$g(i+1)=b_u+j^*<b_{u+1}$. 
We have 
\begin{equation} \label{equation b} 
f(i+1)-f(i)=(-1)^{s+u}(g(i+1)-g(i))=(-1)^{s+u}(j^*-j) 
\end{equation} 
by Lemma \ref{unfolding induces reduction}. 
Moreover, by the induction hypothesis $(*_i)$, we have $f(i)=f(a)+(-1)^s (\Sigma_u+(-1)^u j)$ and hence $(*_{i+1})$ by equation (\ref{equation b}).

It remains to prove the second claim. Note that by the first claim, $f g^{-1}$ defines a unique map on $\ran{g}$. 
Moreover, it follows immediately from the first claim that the restrictions of $f g^{-1}$ to sub-intervals of $\ran{g}$ of the form $[n_i,n_{i+1})$ are affine functions whose direction changes between adjacent intervals. 
Since $\Delta {\bf n}$ is strictly increasing,  the lengths $p_u=(b_{u+1}-b_u)-1$ of the intervals $[b_u,b_{u+1}-1]$ have the same property. 
It thus follows by induction on $u$ that for all $u$ with $t\leq u< w$, we have 
$$\ran{f g^{-1}{\upharpoonright}[g(a),b_{u+1}-1]}=\ran{f g^{-1}{\upharpoonright}[b_u,b_{u+1}-1]}$$
Since the length of $[b_v,b_{v+1}-1]$ is $p_v$, this implies the first distance statement. Moreover, by comparing the lengths of the last two intervals, we obtain 
$$\ran{f g^{-1}{\upharpoonright}[g(a),b_{v+1}-1)}=\ran{f g^{-1}{\upharpoonright}[b_v,b_{v+1}-1)}$$
Since the length of $[b_v,b_{v+1}-1)$ is $p_v-1$, the last distance statement follows. 
\end{proof} 

We are now ready to prove Lemma \ref{unfoldings with large domain implies tail equivalence} above, which we restate now. 

\begin{lemma} \label{tail equivalence from unfoldings} 
Suppose that $\bf m$, $\bf n$ are strictly increasing sequences of natural numbers beginning with $0$ such that also $\Delta \bf m$, $\Delta \bf n$ are strictly increasing. Moreover, suppose that 
$$\bigcup_{(f,g)\in \mathbb{U}_l^{{\bf m}, {\bf n}}} \ran{f}=\NN$$ 
for some $l\geq 0$. 
Then $\Delta\bf m$, $\Delta\bf n$ are $E_{\mathrm{tail}}$-equivalent. 
\end{lemma} 
\begin{proof} 
Let $p_u=2(m_{u+1}-m_u)-1$ and $q_v=2(n_{v+1}-n_v)-1$ for $u,v\geq 0$. 
Moreover, suppose that $$l\in [2n_t,2n_{t+1}-1]$$ and let 
\[ \Phi=\{p_u\mid u\geq 0,\  p_u>p_0, q_t\}\] 
\[ \Psi=\{q_v\mid v\geq 0,\  q_v>p_0, q_t\}.\] 

\begin{claim*} 
If $u\geq 0$, $p_u\in \Phi$ and $v$ is least with $p_u\leq q_v$, then $u,v$ have the same parity, i.e. both are even or both are odd. 
\end{claim*} 
\begin{proof} 
Let 
$$a=2m_u+p_u=2m_{u+1}-1$$ 
and 
$$b=2n_v+p_u.$$ 


By the assumptions, there is some unfolding $\xi=(f,g)\in \mathbb{U}_l^{{\bf m},{\bf n}}$ with $a\in\ran{f}$. 
Suppose that $f$ is a reduction from $I$ to $G_{\bf m}$ and $g$ is a reduction from $I$ to $G_{\bf n}$, where $I$ is a finite colored graph. 

Since $\xi\in\mathbb{U}_l^{{\bf m},{\bf n}}$, there is some $j_0\in\dom{f}$ with $f(j_0)=0$ and $g(j_0)=l$.  Moreover, since $a\in\ran{f}$, there is some $k\in\dom{f}$ with $f(k)=a$. 
We can assume that $k>j_0$, since the case that $k<j_0$ is symmetric (by changing the order of $j_0$, $j$, $j_*$, $j^*$, $k$ below). 

The next subclaim follows almost immediately from the definitions. 

\begin{subclaim*} 
$f(j_0)<2m_u\leq a$ and $g(j_0)<2n_v\leq b$. 
\end{subclaim*} 
\begin{proof} 
For the first claim, the assumption $p_u\in\Phi$ implies that $p_0< p_u$. Since $\Delta{\bf m}$ is strictly increasing, the last inequality implies that $0<u$ and hence $0\leq2m_{1}-1<2m_u$. Therefore, 
\[ f(j_0)=0< 2m_u\leq 2m_u+p_u=a.\] 

The proof of the second claim is very similar to the previous proof. 
The assumptions $p_u\in\Phi$ and $p_u\leq q_v$ imply that $q_t<p_u\leq q_v$. Since $\Delta{\bf n}$ is strictly increasing, the last inequality implies that $t<v$ and hence $l\leq2n_{t+1}-1<2n_v$. Therefore, 
\[ g(j_0)=l< 2n_v\leq 2n_v+p_u=b.\] 
\end{proof} 

Let $j^*>j_0$ be least with $f(j^*)=a$ or $g(j^*)=b$. We can assume that $f(j^*)=a$, since the case that $g(j^*)=b$ is symmetric (by replacing $f$, $a$, $m_u$, $p_u$ with $g$, $b$, $n_v$, $q_v$ and conversely). 

\begin{subclaim*} 
$g(j^*)=b$. 
\end{subclaim*} 
\begin{proof} 
Suppose that $g(j^*)\neq b$. By the choice of $j^*$ and since $g(j_0)<b$ by the previous subclaim, we have $g(j)<b$ for all $j\in[j_0,j^*]$. 

Let $j<j^*$ be minimal such that $f(i)\in [2m_u,a]$ for all $i\in[j,j^*]$. 
Since $f(j_0)<2m_u$ by the previous subclaim, the choice of $j$ implies that $j>j_0$ and $f(j)=2m_u$. 
We further have $q_w<p_u$ for all $w<v$ by the choice of $v$, and moreover $\ran{g{\upharpoonright}[j_0,j^*]}$ is bounded strictly below $b=2n_v+p_u$, as we argued above. 

By the definition of $j^*$, we can now apply Lemma \ref{sum formula} to restrictions of $f$ to sub-intervals of $[j,j^*]$. 
Let $j_*<j^*$ be maximal with $g(j_*)=g(j)$. Since the last equation holds for $j_*=j$, we have $j\leq j_*$. 

We first assume that $j<j_*$. 
We then choose some $a^*\in [j,j_*]$ with $g(a^*)$ least among all $g(i)$ for $i\in[j,j_*]$. 
We further consider the unique $t^*$, $w$ such that the minimal and maximal values of $g(i)$ for $i\in[j,j_*]$ are elements of $[2n_{t^*},2n_{t^*+1}-1]$ and $[2n_w,2n_{w+1}-1]$, respectively. 
The sum formula in the first claim of Lemma \ref{sum formula} can now be applied to the unfolding that is given by the restrictions of $f$, $g$ to $[j,j_*]$ instead of $f$, $g$ and to  the parameters $u$, $t^*$, $w$, $a^*$ instead of $s$, $t$, $w$, $a$ above. Since $g(j_*)=g(j)$ by the choice of $j_*$, it follows immediately from the sum formula that $f(j_*)=f(j)=2m_u$, similar to the argument in the end of the proof of Lemma \ref{sum formula}. 

We now choose some $a_*\in [j_*,j^*]$ with $g(a_*)$ least among all $g(i)$ for $i\in[j_*,j^*]$ and consider the unique $t_*$ with $g(a_*)\in [2n_{t_*},2n_{t_*+1}-1]$. 
We then apply the second claim of Lemma \ref{sum formula} to the unfolding that is given by the restrictions of $f$, $g$ to $[j_*,j^*]$ instead of $f$, $g$ and to the parameters $u$, $t_*$, $v$, $v+1$, $a_*$ instead of $s$, $t$, $v$, $w$, $a$ above. 

To apply this, the following conditions are satisfied. 
The range of the restriction of $g$ to $[j_*,j^*]$ is a subset of $[g(a^*),2n_{v+1}-1)$, since our assumption $p_u\leq q_v$ and the argument in the beginning of the proof of this subclaim imply that the range is strictly bounded by $b=2n_v+p_u\leq 2n_v+q_v=2n_{v+1}-1$. 
Moreover its minimum is $g(a_*)$ by the definition of $a_*$, $t_*$. 

Hence $|f(i)-f(i^*)|<p_u$ for all $i, i^*\in [j_*,j^*]$. However, this contradicts the assumption $f(j^*)=a=2m_u+p_u$ that was made before the current subclaim, and the fact that $f(j_*)=2m_u$ as we have just proved. 

If $j=j_*$, then the argument in the last paragraph leads to the  same contradiction. 
%
\end{proof}

Before we complete the proof, we discuss the changes that have to be made above in the symmetric cases. The proof of the first subclaim will remain unchanged. 

We first assume that $k<j_0$ 
and describe how we will define numbers $j$, $j_*$, $j^*$ with $k<j^*<j_*<j<j_0$ in reverse order compared to above. 
Before the second subclaim, we now let $j^*<j_0$ be largest instead of $j^*>j_0$ least with $f(j^*)=a$ or $g(j^*)=b$. 
In the proof of the second subclaim, we consider the interval $[j^*,j_0]$ instead of $[j_0,j^*]$, choose $j>j^*$ maximal instead of $j<j^*$ minimal and $j_*>j^*$ minimal instead of $j_*<j^*$ maximal with the same properties as above. 
If $j>j^*$, we first apply Lemma \ref{sum formula} to the restrictions of $f$, $g$ to the interval $[j^*,j]$ and then to the interval $[j^*,j_*]$ with the same parameters as above. 
If $j=j^*$, then as above, the argument for the last application of Lemma \ref{sum formula} is sufficient. 

We now consider the changes for the symmetric case $g(j^*)=b$ to the assumption $f(j^*)=a$ made before the second subclaim in the case $k>j_0$. 
This is obtained by replacing $f$, $a$, $m_u$, $p_u$ with $g$, $b$, $n_v$, $q_v$ and conversely as follows. 
The changes for the case $k<j_0$ are then obtained by again changing the order of $j_0$, $j$, $j_*$, $j^*$, $k$, as we have just described. 

The statement of the second subclaim is now that $f(j^*)=a$. In the proof of the second subclaim, we assume towards a contradiction that $f(j^*)\neq a$. It follows as above that $f(j)<a$ for all $j\in[j_0,j^*]$. 
We then work with the condition $g(i)\in [2n_v,b]$ instead of $f(i)\in [2m_u,a]$. 
In the proof above, we used that $q_w<p_u$ for all $w<v$ by the choice of $v$. Since $p_u\leq q_v$ by the choice of $v$ and $p_w<p_u$ for all $w<u$, we now have the symmetric property that $p_w<q_v$ for all $w<u$. 

In the following, the role of $a^*$ is replaced with $b^*$, where $b^*$ is defined using $f$ instead of $g$. Moreover, the role of $f$ and $g$ are interchanged here and in the rest of the proof. 
Lemma \ref{sum formula} is now applied to the unfolding that is given by the restrictions of $g$, $f$ to $[j,j_*]$ instead of $f$, $g$ and to the parameters $v$, $t^*$, $w$, $b^*$ 
instead of $s$, $t$, $w$, $a$. 
We then define $b_*$ via $f$ instead of $a_*$ via $g$. 
The second application of Lemma \ref{sum formula} is to the unfolding that is given by the restrictions of $g$, $f$ to $[j_*,j^*]$ instead of $f$, $g$ and to the parameters $v$, $t_*$, $u$, $u+1$, $b_*$ 
instead of $s$, $t$, $v$, $w$, $a$. 

This completes the description of the symmetric cases and we now complete the proof of the claim.

Since $f(j^*)=a$ by the assumption and $g(j^*)=b$ by the previous subclaim, we have 
\[c_{\bf m}(a)=c_{\bf m}(f(j^*))=c_I(j^*)=c_{\bf n}(g(j^*))=c_{\bf n}(b).\] 
We have $a=2m_u+p_u<2m_{u+1}$ and $b=2n_v+p_u<2n_{v+1}$, since $p_u\leq q_v$ by the assumptions of the claim. 
Thus the last equation and the definition of $G_{\bf m}$, $G_{\bf n}$ imply that $c_{\bf m}(2m_u)=c_{\bf n}(2n_v)$. 
By the definition of $G_{\bf m}$, $G_{\bf n}$ in Definition \ref{definition of Gn}, this implies that $m_u\in \Even_{\bf m}\Longleftrightarrow n_v\in \Even_{\bf n}$. 
However, $m_u\in \Even_{\bf m}$ holds if and only if $u$ is even by the definition of $\Even_{\bf m}$ in Definition \ref{definition: even and odd blocks}, and the same holds for $n_v$, $\Even_{\bf n}$ and $v$. Hence $u$, $v$ have the same parity. 
\end{proof} 

We now let  
$$u_0=\min\{u\geq 0\mid p_u\in \Phi\}$$ 
$$v_0=\min\{v\geq 0\mid q_v\in \Psi\}.$$ 

\begin{claim*} 
$\Delta{\bf m}$, $\Delta{\bf n}$ are $E_{\mathrm{tail}}$-equivalent. 
\end{claim*} 
\begin{proof} 
It follows immediately from the definitions of $\Phi$ and $\Psi$ that the sequences 
$${\bf p}=\langle p_u\mid u\geq u_0\rangle$$ 
$${\bf q}=\langle q_v\mid v\geq v_0\rangle$$ 
enumerate $\Phi$ and $\Psi$, respectively. 
By the definitions of $p_u$ and $q_v$ and since $\Delta{\bf m}$, $\Delta{\bf n}$ are strictly increasing by the assumption, the sequences $\bf p$, $\bf q$ are strictly increasing. 

We now show that $\Phi=\Psi$. By the properties of $\bf p$, $\bf q$ that were just stated, this implies that $p_{u_0+j}=q_{v_0+j}$ for all $j\geq 0$, and hence $\Delta {\bf m}$, $\Delta{\bf n}$ are $E_{\mathrm{tail}}$-equivalent, proving the claim. We will only show that $\Phi\subseteq \Psi$, since the proof of the other inclusion is symmetric. 

Suppose that $p_u\in\Phi$ and $v$ is least with $p_u\leq q_v$. Then $u$, $v$ have the same parity by the first claim. 

\begin{subclaim*} 
$q_v<p_{u+1}$. 
\end{subclaim*} 
\begin{proof} 
Towards a contradiction, suppose that $p_{u+1}\leq q_v$. 
By the choice of $v$ as least with $p_u\leq q_v$, this implies that $v$ is also least with $p_{u+1}\leq q_v$. 
Then $u+1$, $v$ have the same parity by the first claim. But this contradicts the fact that $u$, $v$ have the same parity. 
\end{proof} 

To complete the proof of the claim, it is sufficient to show that $p_u=q_v$. Towards a contradiction, suppose that $p_u<q_v$. Since $p_u<q_v<p_{u+1}$ by the previous subclaim, $w=u+1$ is least with $q_v<p_w$. 
Note that the proof of the first claim also works for $\Psi$, $\bf n$, $\bf m$ instead of $\Phi$, $\bf m$, $\bf n$ by switching the roles of $p_0$ and $q_t$. 
Therefore $u+1$, $v$ have the same parity. 
However, this contradicts the fact that $u$, $v$ have the same parity. 
\end{proof} 

The statement of the previous claim 
completes the proof of Lemma \ref{tail equivalence from unfoldings}. 
\end{proof}


\subsection{The compatibility range 
is open and closed
}
\label{section proofs Xi is open and closed} 
In this section we will prove Lemma \ref{Xi is open intro} and 
two auxiliary results. 

As stated before Definition \ref{definition compatible}, we always assume the following situation. 
We assume that $\alpha$ is a positive and $\beta$ is an arbitrary element of $X$. 
The sequences $\bf m$, $\bf n$ are strictly increasing sequences in $\NN$ beginning with $0$ as above and $F\colon X\rightarrow X$ is a continuous function with $D^\alpha_{\bf m}\cap Y^{\alpha,\beta}=F^{-1}[D^\beta_{\bf n}]\cap Y^{\alpha,\beta}$, where $Y^{\alpha,\beta}=X^{\alpha}\cap F^{-1}[X^{\beta}]$. 
From now on, we additionally assume that $\Delta{\bf m}$, $\Delta{\bf n}$ are strictly increasing.


\begin{definition} \label{definition: small} 
Suppose that $\alpha$, $\beta$, $F$ are as above and $x\in Y^{\alpha,\beta}$. 
A subset $U$ of $Y^{\alpha,\beta}$ is called \emph{$(\alpha,\beta)$-small at $x$} if $x\in U$ and the following conditions hold for all $i,j\in \mathbb{N}$ with $i<j$ and 
$A=(i,j)$. 
\begin{enumerate-(a)} 
\item 
If $x\in C_A^\alpha$, then $U\subseteq C_A^\alpha$. 
\item 
If $F(x)\in C_A^\beta$, then $F[U]\subseteq C_A^\beta$. 
\end{enumerate-(a)} 
\end{definition} 

The next lemma is almost immediate from the previous definition. 

\begin{lemma} \label{existence of small open set} 
Suppose that $\alpha$, $\beta$ are as above, $F\colon X\rightarrow X$ is continuous and $x\in Y^{\alpha,\beta }$. 
Then there is an open subset $U$ of $Y^{\alpha,\beta}$ 
that is $(\alpha,\beta)$-small at $x$. 
\end{lemma} 
\begin{proof} 
First, we note that $B^{\alpha}_A$ as given in Definition \ref{definition of C_A} is an open subset of $X$, whenever $A$ is an open interval in $\RR_{\geq0}$. 
By the definition of the sets $C_A^{\alpha}$, we further have that  
$X^{\alpha}$ is the union of all sets $C_{(i,j)}^\alpha$, where 
$0\leq i<j$. 
Hence we can associate to each $x\in X^{\alpha}$ the smallest interval $A_x^\alpha=(i,j)$ with 
$0\leq i<j$ and $x\in C_{(i,j)}^\alpha$. 
Then $U=C_{A_x^\alpha}^\alpha\cap F^{-1}[C_{A_{F(x)}^\beta}^\beta]$ is $(\alpha,\beta)$-small at $x$. Moreover $U$ is open, since $F$ is continuous. 
\end{proof}

If $\beta$ is as above, let $\Gamma^\beta=B^{\alpha}_{A^\beta}$, where $A^\beta=\{r_t^\alpha\mid t\geq 1\}$. 
The next lemma reduces the later proofs to the case that $x\in \Gamma^\alpha$ and $F(x)\in \Gamma^\beta$.

\begin{lemma} \label{same unfolding} 
Suppose that $\bf m$, $\bf n$, $\alpha$, $\beta$, $F$ are as above, 
$x\in Y^{\alpha,\beta}$ and $x\notin \Gamma^\alpha$ or $F(x)\notin \Gamma^\beta$. If $U$ is an open subset of $Y^{\alpha,\beta}$ that is $(\alpha,\beta)$-small at $x$ and $y\in U$, then 
\[x\in X_{{\bf m},{\bf n},\xi}^{\alpha,\beta,F} \Longleftrightarrow y\in X_{{\bf m},{\bf n},\xi}^{\alpha,\beta,F}\] 
for all unfoldings $\xi$. 
\end{lemma} 
\begin{proof} 
Note that the assumptions imply that $x \in Y^{\alpha,\beta}$. 
We will distinguish the following cases, depending on the values of $x$ and $F(x)$. 
\begin{case*} 
$x\notin \Gamma^\alpha$ and $F(x)\notin \Gamma^\beta$. 
\end{case*} 
\begin{proof} 
We can assume that $x\in D_{\bf m}^\alpha$, since the case that $x\in E_{\bf m}^\alpha$ is symmetric. Since $x\in  D_{\bf m}^\alpha\setminus \Gamma^\alpha$, we have $x\in C_{(2j,2j+1)}^\alpha$ for some $j\geq 0$, and by the assumption that $U$ is $(\alpha,\beta)$-small at $x$, this implies that $U\subseteq C_{(2j,2j+1)}^\alpha$.

Since $x\in D^\alpha_{\bf m}\cap Y^{\alpha,\beta}=F^{-1}[D^\beta_{\bf n}]\cap Y^{\alpha,\beta}$, we have $F(x)\in D_{\bf n}^\beta$ and thus $F(x)\in D_{\bf n}^\beta\setminus \Gamma^\beta$ by the case assumption. Hence $F(x)\in C_{(2k,2k+1)}^\beta$ for some $k\geq 0$, and by the assumption that $U$ is $(\alpha,\beta)$-small at $x$, this implies that $F[U]\subseteq C_{(2k,2k+1)}^\beta$. 
The equivalence follows from the inclusions $U\subseteq C_{(2j,2j+1)}^\alpha$ and $F[U]\subseteq C_{(2k,2k+1)}^\beta$. 

We now indicate the changes to the previous proof in the remaining case that $x\in E_{\bf m}^\alpha$. Here we simply replace $D_{\bf m}^\alpha$, $D_{\bf n}^\beta$ with $E_{\bf m}^\alpha$, $E_{\bf n}^\beta$ and 
shift all indices upwards by $1$. 
\end{proof} 

\begin{case*} 
$x \notin \Gamma^\alpha$ and $F(x)\in \Gamma^\beta$. 
\end{case*} 
\begin{proof} 
We can assume that $x\in D_{\bf m}^\alpha$, since the case that $x\in E_{\bf m}^\alpha$ is symmetric. Since $x\in  D_{\bf m}^\alpha\setminus \Gamma^\alpha$, we have $x\in C_{(2j,2j+1)}^\alpha$ for some $j\geq 0$, and by the assumption that $U$ is $(\alpha,\beta)$-small at $x$, this implies that $U\subseteq C_{(2j,2j+1)}^\alpha$. 

By the assumption that $F(x)\in \Gamma^\beta$, we can assume that $d(\beta,F(x))=r_i^\beta$ and $i=2k$ for some $k\geq 1$, since the case that $i=2k+1$ for some $k\geq 0$ is symmetric. 
Since $U$ is $(\alpha,\beta)$-small at $x$ and $d(\beta,F(x))=r_{2k}^\beta$, we have $F[U]\subseteq C_{(2k-1,2k+1)}^\beta$. Therefore 
$$U\subseteq D^\alpha_{\bf m}\cap Y^{\alpha,\beta}=F^{-1}[D^\beta_{\bf n}]\cap Y^{\alpha,\beta}$$ 
and hence 
$$F[U]\subseteq C_{(2k-1,2k+1)}^\beta \cap D_{\bf n}^\beta=C_{[2k,2k+1)}^\beta$$ 
(note that in this case $k\in \mathrm{Even}_{\bf n}$, since $d(\beta,F(x))=r_{2k}^\beta$ and $F(x)\in D_{\bf n}^\beta$). 
The equivalence follows from the inclusions $U\subseteq C_{(2j,2j+1)}^\alpha$ and $F[U]\subseteq C_{[2k,2k+1)}^\beta$. 

We again indicate the necessary changes in the remaining cases. If $x\in E_{\bf m}^\alpha$, we replace $D_{\bf m}^\alpha$, $D_{\bf n}^\beta$ by $E_{\bf m}^\alpha$, $E_{\bf n}^\beta$ and $(2j,2j+1)$, $[2k,2k+1)$ by $(2j+1,2j+2)$, $(2k-1,2k)$ in the previous proof. 
If $i=2k+1$, we replace $(2k-1,2k+1)$, $[2k,2k+1)$ by $(2k,2k+2)$, $(2k,2k+1)$. 
Finally, if both $x\in E_{\bf m}^\alpha$ and $i=2k+1$, we replace $D_{\bf m}^\alpha$, $D_{\bf n}^\beta$ by $E_{\bf m}^\alpha$, $E_{\bf n}^\beta$ and shift all indices upwards by $1$. 
\end{proof} 

\begin{case*} 
$x\in \Gamma^\alpha$ and $F(x) \notin \Gamma^\beta$. 
\end{case*} 
\begin{proof} 
We can assume that $x\in D_{\bf m}^\alpha$, since the case $x\in E_{\bf m}^\alpha$ is symmetric. Since 
$D^\alpha_{\bf m}\cap Y^{\alpha,\beta}=F^{-1}[D^\beta_{\bf n}]\cap Y^{\alpha,\beta}$, 
this implies that $F(x)\in D_{\bf n}^\beta$, and hence $F(x)\in D_{\bf n}^\beta\setminus \Gamma^\beta$ by the assumption. 
Then $F(x)\in C_{(2k,2k+1)}^\beta$ for some $k\geq 0$, and since $U$ is $(\alpha,\beta)$-small at $x$, this implies that $F[U]\subseteq C_{(2k,2k+1)}^\beta$. 

By the assumption that $x\in \Gamma^\alpha$, we can assume that $d(\alpha,x)=r_i^\alpha$ and $i=2j$ for some $j\geq 1$, since the case that $i=2j+1$ for some $j\geq 0$ is symmetric. Since $U$ is $(\alpha,\beta)$-small at $x$ and $d(\alpha,x)=r_{2j}^\alpha$, we have $U\subseteq C_{(2j-1,2j+1)}^\alpha$. Since moreover 
$$F[U]\subseteq C_{(2k,2k+1)}^\beta\subseteq D_{\bf n}^\beta,$$
we have $U\subseteq F^{-1}[D^\beta_{\bf n}]\cap Y^{\alpha,\beta}=D_{\bf m}^\alpha \cap Y^{\alpha,\beta}$. 
Hence 
$$U\subseteq C_{(2j-1,2j+1)}^\alpha\cap D_{\bf m}^\alpha=C_{[2j,2j+1)}^\alpha$$
(note that in this case $j\in \mathrm{Even}_{\bf m}$, since $d(\alpha,x)=r_{2j}^\alpha$ and $x\in D_{\bf m}^\alpha$). 
The equivalence follows from the inclusions $U\subseteq C_{[2j,2j+1)}^\alpha$ and $F[U]\subseteq C_{(2k,2k+1)}^\beta$. 

We finally indicate the changes in the remaining cases. If $x\in E_{\bf m}^\alpha$, we replace $D_{\bf m}^\alpha$, $D_{\bf n}^\beta$ by $E_{\bf m}^\alpha$, $E_{\bf n}^\beta$ and $(2k,2k+1)$, $[2j,2j+1)$ by $(2k+1,2k+2)$, $(2j-1,2j)$. 
If $i=2j+1$, we replace $(2j-1,2j+1)$, $[2j,2j+1)$ by $(2j,2j+2)$, $(2j,2j+1)$. 
Finally, if both $x\in E_{\bf m}^\alpha$ and $i=2j+1$, we replace $D_{\bf m}^\alpha$, $D_{\bf n}^\beta$ by $E_{\bf m}^\alpha$, $E_{\bf n}^\beta$ and shift all indices upwards by $1$. 
\end{proof} 
Since we have covered all cases, the conclusion of Lemma \ref{same unfolding} follows. 
\end{proof}

We are now ready to prove the first part of Lemma \ref{Xi is open intro} above, which we restate now. 

\begin{lemma} \label{Xi is open} 
For all $\bf m$, $\bf n$, $\alpha$, $\beta$, $F$ as above, $s\geq 0$ and $t\geq1$, $X_{{\bf m}, {\bf n}, s, <t}^{\alpha,\beta,F}$ is an open subset of $X^{\alpha}_{<t}$. 
\end{lemma} 
\begin{proof} 
Suppose that $x\in X_{{\bf m}, {\bf n}, s, <t}^{\alpha,\beta,F}\subseteq Y^{\alpha,\beta}$. 
By Lemma \ref{existence of small open set}, there is an open subset $U$ of $Y^{\alpha,\beta}$ that is $(\alpha,\beta)$-small at $x$ and it is sufficient to show that $U\subseteq X_{{\bf m}, {\bf n}, s, <t}^{\alpha,\beta,F}$. 

The claim follows from Lemma \ref{same unfolding} if $x\notin \Gamma^\alpha$ or $F(x)\notin \Gamma^\beta$. Therefore, we can assume that $x\in \Gamma^\alpha$ and $F(x)\in \Gamma^\beta$, and thus $d(\alpha,x)=r_k^\alpha$ and $d(\beta,F(x))=r_l^\beta$ for some $k,l\geq 1$. 
Since $x\in X_{{\bf m}, {\bf n}, s, <t}^{\alpha,\beta,F}\subseteq X^\alpha_{<t}$, this implies that $t>k$.

We can further assume that $k=2i$ and $l=2j$ for some $i,j\geq 1$, since the cases that $k=2i+1$ or $l=2j+1$ for some $j\geq0$ are symmetric. We can finally assume that $x\in D_{\bf m}^\alpha$, since the case that $x\in E_{\bf m}^\alpha$ is symmetric. 

Since $d(\alpha,x)=r_{2i}^\alpha$, the last assumption implies that $i\in\mathrm{Even}_{\bf m}$. Since $x\in D^\alpha_{\bf m}\cap Y^{\alpha,\beta}=F^{-1}[D^\beta_{\bf n}]\cap Y^{\alpha,\beta}$, 
we further have $F(x)\in D_{\bf n}^\beta$, and since $d(\beta,F(x))=r_{2j}^\beta$, this implies that $j\in\mathrm{Even}_{\bf n}$. It now follows from $i\in\mathrm{Even}_{\bf m}$ and $j\in\mathrm{Even}_{\bf n}$ that 
$$c_{\bf m}(k)=c_{\bf m}(2i)=c_{\bf n}(2j)=c_{\bf n}(l).$$ 
Since $x\in X_{{\bf m}, {\bf n}, s, <t}^{\alpha,\beta,F}$, there is some unfolding $\xi=(f,g)$ with $(k,k+1)\sim_{\xi}(l,l+1)$ and $0\sim_{\xi}s$. 
Now the first case of Lemma \ref{extend unfolding} yields an unfolding $\xi^*=(f^*,g^*)$ of $G_{\bf m}$, $G_{\bf n}$ 
that extends $\xi$ with $(k-1,k)\sim_{\xi^*}(l-1,l)$ and $\ran{f^*}=\ran{f}\cup\{k\}$. 
We have $0\sim_{\xi^*}s$, since $\xi^*$ extends $\xi$. 
Since moreover $t\notin \ran{f}$ and $t>k$, we have $t\notin \ran{f^*}$. 
It is thus sufficient to show the following. 

\begin{claim*} 
$U\subseteq X_{{\bf m}, {\bf n}, \xi^*}^{\alpha,\beta,F}$. 
\end{claim*} 
\begin{proof} 
Suppose that $y\in U$ is given. 
Since $U$ is $(\alpha,\beta)$-small at $x$ and moreover $d(\alpha,x)=r_{2i}^\alpha$ and $d(\beta,F(x))=r_{2j}^\beta$, we have $U\subseteq C_{(2i-1,2i+1)}^\alpha$ and $F[U]\subseteq C_{(2j-1,2j+1)}^\beta$. In particular, we have that $y\in Y^{\alpha,\beta}$. 

\begin{case*} 
$y\in C_{(2i-1,2i)}^\alpha$. 
\end{case*} 

Since $D^\alpha_{\bf m}$, $E^\alpha_{\bf m}$ are complements in $X^{\alpha*}$ and $D^\beta_{\bf n}$, $E^\beta_{\bf n}$ in $X^{\beta*}$, our assumption $D^\alpha_{\bf m}\cap Y^{\alpha,\beta}=F^{-1}[D^\beta_{\bf n}]\cap Y^{\alpha,\beta}$ immediately implies that $E^\alpha_{\bf m}\cap Y^{\alpha,\beta}=F^{-1}[E^\beta_{\bf n}]\cap Y^{\alpha,\beta}$. 
Since $C_{(2i-1,2i)}^\alpha\subseteq E_{\bf m}^\alpha$, we then have $y\in E^\alpha_{\bf m}\cap Y^{\alpha,\beta}=F^{-1}[E^\beta_{\bf n}]\cap Y^{\alpha,\beta}$ and hence $F(y)\in E_{\bf n}^\beta$. 
Since $j\in \mathrm{Even}_{\bf n}$, we thus have 
$$F(y)\in C_{(2j-1,2j+1)}^\beta\cap E_{\bf n}^\beta=C_{(2j-1,2j)}^\beta.$$
Since $(2i-1,2i)\sim_{\xi^*}(2j-1,2j)$ by the choice of $\xi^*$, 
it follows that $y\in X_{{\bf m}, {\bf n}, \xi^*}^{\alpha,\beta,F}$. 

\begin{case*} 
$y\in C_{[2i,2i+1)}^\alpha$. 
\end{case*} 

Since $C_{[2i,2i+1)}^\alpha\subseteq D_{\bf m}^\alpha$, then $y\in D^\alpha_{\bf m}\cap Y^{\alpha,\beta}=F^{-1}[D^\beta_{\bf n}]\cap Y^{\alpha,\beta}$ and hence $F(y)\in D^\beta_{\bf n}$. 
Since $j\in \mathrm{Even}_{\bf n}$, we have 
$$F(y)\in C_{(2j-1,2j+1)}^\beta\cap D_{\bf n}^\beta=C_{[2j,2j+1)}^\beta.$$
Since $(2i,2i+1)\sim_{\xi}(2j,2j+1)$ by the choice of $\xi$, 
it follows that $y\in X_{{\bf m}, {\bf n}, \xi}^{\alpha,\beta,F}\subseteq X_{{\bf m}, {\bf n}, \xi^*}^{\alpha,\beta,F}$. 
\end{proof} 

The previous claim completes the proof of Lemma \ref{Xi is open} in the case that is stated above. 
We now describe the necessary changes for the remaining cases. 

We first assume $x\in E_{\bf m}^\alpha$ instead of $x\in D_{\bf m}^\alpha$.  We then have $i\in \Odd_{\bf m}$, $j\in\Odd_{\bf n}$, $(k-1,k)\sim_\xi (l-1,l)$ and obtain some $\xi^*$ with $(k,k+1)\sim_{\xi^*}(l,l+1)$ by the second case of Lemma \ref{extend unfolding}. 
As in the two cases in the previous proof, it follows that $y\in C_{(2i,2i+1)}^\alpha\Longleftrightarrow F(y)\in C_{(2j,2j+1)}^\beta$ and therefore $y\in X_{{\bf m}, {\bf n}, \xi^*}^{\alpha,\beta,F}$ for all $y\in U$. 

Second, we assume that $k=2i+1$ and $l=2j+1$. If $x\in D_{\bf m}^\alpha$, then $i\in \Odd_{\bf m}$, $j\in\Odd_{\bf n}$, $(k-1,k)\sim_\xi (l-1,l)$ and there is some $\xi^*$ with $(k,k+1)\sim_{\xi^*}(l,l+1)$ by the second case of Lemma \ref{extend unfolding}. 
Similarly to the above proof, we obtain $y\in C_{(2i+1,2i+2)}^\alpha\Longleftrightarrow F(y)\in C_{(2j+1,2j+2)}^\beta$ and the conclusion follows. 
On the other hand, if $x\in E_{\bf m}^\alpha$, then $i\in \Even_{\bf m}$, $j\in\Even_{\bf n}$, $(k,k+1)\sim_\xi (l,l+1)$. We then obtain some $\xi^*$ with $(k-1,k)\sim_{\xi^*}(l-1,l)$ by the first case of Lemma \ref{extend unfolding} and $y\in C_{(2i,2i+1)}^\alpha\Longleftrightarrow F(y)\in C_{(2j,2j+1)}^\beta$ for all $y\in U$. 

Third, we assume that $k=2i$ and $l=2j+1$. If $x\in D_{\bf m}^\alpha$, then $i\in \Even_{\bf m}$, $j\in\Odd_{\bf n}$, $(k,k+1)\sim_\xi (l-1,l)$ and we obtain some $\xi^*$ with $(k-1,k)\sim_{\xi^*}(l,l+1)$ by the third case of Lemma \ref{extend unfolding}. 
As above, it follows that $y\in C_{(2i-1,2i)}^\alpha\Longleftrightarrow F(y)\in C_{(2j+1,2j+2)}^\beta$ for all $y\in U$. 
If $x\in E_{\bf m}^\alpha$, then $i\in \Odd_{\bf m}$, $j\in\Even_{\bf n}$, $(k-1,k)\sim_\xi (l,l+1)$. We then obtain some $\xi^*$ with $(k,k+1)\sim_{\xi^*}(l-1,l)$ by the last case of Lemma \ref{extend unfolding} and $y\in C_{(2i,2i+1)}^\alpha\Longleftrightarrow F(y)\in C_{(2j,2j+1)}^\beta$ for all $y\in U$. 

Finally, assume that $k=2i+1$ and $l=2j$. If $x\in D_{\bf m}^\alpha$, then $i\in \Odd_{\bf m}$, $j\in\Even_{\bf n}$, $(k-1,k)\sim_\xi (l,l+1)$ and we obtain some $\xi^*$ with $(k,k+1)\sim_{\xi^*}(l-1,l)$ by the last case of Lemma \ref{extend unfolding}. 
Then $y\in C_{(2i+1,2i+2)}^\alpha\Longleftrightarrow F(y)\in C_{(2j-1,2j)}^\beta$ for all $y\in U$. 
If $x\in E_{\bf m}^\alpha$, then $i\in \Even_{\bf m}$, $j\in\Odd_{\bf n}$, $(k,k+1)\sim_\xi (l-1,l)$, we obtain $\xi^*$ with $(k-1,k)\sim_{\xi^*}(l,l+1)$ by the third case of Lemma \ref{extend unfolding} and $y\in C_{(2i,2i+1)}^\alpha\Longleftrightarrow F(y)\in C_{(2j,2j+1)}^\beta$ for all $y\in U$. 
\end{proof} 

We are now ready to prove the second part of Lemma \ref{Xi is closed intro} above, which we restate now.

\begin{lemma} \label{Xi is closed} 
For all $\bf m$, $\bf n$, $\alpha$, $\beta$, $F$ as above, $s\geq 0$ and $t\geq1$, $X_{{\bf m}, {\bf n}, s, <t}^{\alpha,\beta,F}$ is a relatively closed subset of $X^{\alpha}_{<t}$. 
\end{lemma} 
\begin{proof} 
Suppose that $x\in X^\alpha_{<t}$ and $x=\lim_{u\geq 0} x_u$ for some sequence ${\bf x}=\langle x_u\mid u\geq 0\rangle$ with $x_u\in X_{{\bf m}, {\bf n}, s, <t}^{\alpha,\beta,F}$ for all $u\geq 0$. 

\begin{claim*} 
$x\in Y^{\alpha,\beta}$. 
\end{claim*} 
\begin{proof} 
We will show that $F(x_u)\in X^\beta_{<k}$ for some $k\geq 1$ and all $u\geq 0$. Since $F$ is continuous, this implies that $F(x)\in X^\beta_{<k+1}$ and thus $x\in Y^{\alpha,\beta}$.

Towards a contradiction, suppose that this claim is false. 
We first choose some $v\geq 0$ with $2m_{v+1}>t$ and let $p_v =2(m_{v+1}-m_v)-1$. 
Moreover, let $k\geq 0$ with $2n_k> s$ and $n_{k+1}-n_k>p_v+1$. 
By the assumption, there is some $u\geq 0$ with $F(x_u)\notin X^\beta_{<2n_{k+1}}$. 
Since $x_u\in X_{{\bf m}, {\bf n}, s, <t}^{\alpha,\beta,F}$, there is some unfolding $\zeta=(f,g)$ of $G_{\bf m}$, $G_{\bf n}$ with $0\sim_\zeta s$, $t\notin \ran{f}$ and $x_u\in X_{{\bf m}, {\bf n},\zeta}^{\alpha,\beta,F}$. 
In particular, $\ran{f}\subseteq \{0,\dots,t-1\}$ and 
$x_u$, $\zeta$ are compatible with respect to ${\bf m}$, ${\bf n}$, $\alpha$, $\beta$, $F$. 
Since $F(x_u)\notin X^\beta_{<2n_{k+1}}$ and by the definition of compatibility, $\ran{g}$ contains a number that is at least $2n_{k+1}$. 
Since $s<2n_k+p_v+1<2n_{k+1}$ by the choice of $k$ and $\ran{g}$ is an interval in $\ZZ$ containing $s$, it follows that $2n_k+p_v\in \ran{g}$. Let $j\in \dom{\xi}$ with $g(j)=2n_k+p_v+1$.

Now let $[i,i^*]$ be a maximal subinterval of $\dom{\xi}$ that containins $j$ with $\ran{g{\upharpoonright}[i,i^*]}\subseteq[2n_k,2n_{k+1}-1]$. 
Since $s\in \ran{g}$ and $s<2n_k$, this must be a strict subinterval and it thus follows from the maximality that $g(i)\in \{2n_k,2n_{k+1}-1\}$ or $g(i^*)\in \{2n_k,2n_{k+1}-1\}$. 
Since $n_{k+1}-n_k>p_v$ and $g(j)=2n_k+p_v+1$ by the choice of $k$ and $j$, we thus have $|g(j)-g(i)|> p_v$ or $|g(j)-g(i^*)|> p_v$.

We further choose some $a\in [i,i^*]$ such that $f(a)$ takes the least value for such $a$. 
Since $2m_{v+1}>t$, there is a unique $l\leq v$ with $f(a)\in [2n_l, 2n_{l+1}-1]$. Let $w=v+1$. 
By \ref{distance statement} of Lemma \ref{sum formula} applied to the restrictions of $g$, $f$ to $[i,i^*]$ instead of $f$, $g$ and the parameters $k$, $l$, $v$, $w$ instead of $s$, $t$, $v$, $w$, we have that $|g(j)-g(j^*)|\leq p_v=2(m_{v+1}-m_v)-1$ for all $j^*\in [i,i^*]$, but this contradicts the inequalities above. 
\end{proof}

By the previous claim and Lemma \ref{existence of small open set}, there is an open subset $U$ of $Y^{\alpha,\beta}$ that is $(\alpha,\beta)$-small at $x$. 
Since $x=\lim_{u\geq 0} x_u\in U$ and $U$ is open, there is some $u\geq 0$ with $x_u\in U$. 
Since moreover $x_u\in X_{{\bf m}, {\bf n}, s, t}^{\alpha,\beta,F}$, there is some unfolding $\xi$ with $x_u\in X_{{\bf m}, {\bf n}, \xi}^{\alpha,\beta,F}$ and $0\sim_\xi u$. 

If $x\notin \Gamma^\alpha$ or $F(x)\notin \Gamma^\beta$, then $x\in X_{{\bf m}, {\bf n}, \xi}^{\alpha,\beta,F}\subseteq X_{{\bf m}, {\bf n}, s, <t}^{\alpha,\beta,F}$ by Lemma \ref{same unfolding}. 
We now assume that $x\in \Gamma^\alpha$ and $F(x)\in \Gamma^\beta$. Then $d(\alpha,x)=r_k^\alpha$ and $d(\beta,F(x))=r_l^\beta$ for some $k,l\geq 1$. 
Since $x\in X_{{\bf m}, {\bf n}, s, <t}^{\alpha,\beta,F}\subseteq X^\alpha_{<t}$, we have $t>k$.

We can further assume that $k=2i$ and $l=2j$ for some $i,j\geq 1$, since the cases that $k=2i+1$ or $l=2j+1$ for some $j\geq0$ are symmetric. 
We can finally assume that $x\in D_{\bf m}^\alpha$, since the case $x\in E_{\bf m}^\alpha$ is symmetric.

Since $d(\alpha,x)=
r_{2i}^\alpha$, 
the last assumption implies that $i\in\mathrm{Even}_{\bf m}$. 
Since moreover 
$F(x)\in Y^{\alpha,\beta}$ by the last claim, we have 
$x\in D^\alpha_{\bf m}\cap Y^{\alpha,\beta}=F^{-1}[D^\beta_{\bf n}]\cap Y^{\alpha,\beta}$ 
and hence $F(x)\in D_{\bf n}^\beta$. 
Since $d(\beta,F(x))=r_{2j}^\beta$, this implies that $j\in\mathrm{Even}_{\bf n}$ and hence 
$$c_{\bf m}(k)=c_{\bf m}(2i)=c_{\bf n}(2j)=c_{\bf n}(l).$$ 
Since $U$ is $(\alpha,\beta)$-small at $x$ and moreover $d(\alpha,x)=r_{2i}^\alpha$ and $d(\beta,F(x))=r_{2j}^\beta$, we have $U\subseteq C_{(2i-1,2i+1)}^\alpha$ and $F[U]\subseteq C_{(2j-1,2j+1)}^\beta$. 
It is sufficient to prove the following.

\begin{claim*} 
$x\in X_{{\bf m}, {\bf n}, s, <t}^{\alpha,\beta,F}$. 
\end{claim*} 
\begin{proof} 
We already argued that $x_u\in U\subseteq C_{(2i-1,2i+1)}^\alpha$. 

\begin{case*} 
$x_u\in D_{\bf m}^\alpha$. 
\end{case*} 
Since $i\in \mathrm{Even}_{\bf m}$, we then have $x_u\in C_{(2i-1,2i+1)}^\alpha\cap D_{\bf m}^\alpha=C_{[2i,2i+1)}^\alpha$. 
Therefore $x_u\in D^\alpha_{\bf m}\cap Y^{\alpha,\beta}=F^{-1}[D^\beta_{\bf n}]\cap Y^{\alpha,\beta}$ and thus $F(x_u)\in D_{\bf n}^\beta$. 
Since moreover $F(x_u)\in F[U]\subseteq C_{(2j-1,2j+1)}^\beta$ and 
$j\in \mathrm{Even}_{\bf n}$, we obtain 
$$F(x_u)\in C_{(2j-1,2j+1)}^\beta\cap D_{\bf n}^\beta=C_{[2j,2j+1)}^\beta.$$ 
Since  $x_u\in X_{{\bf m}, {\bf n}, \xi}^{\alpha,\beta,F}$ and we just argued that $x_u\in C_{[2i,2i+1)}^\alpha$ and $F(x_u)\in C_{[2j,2j+1)}^\beta$, 
we have $(2i,2i+1)\sim_\xi (2j,2j+1)$ by the definition of $\sim_\xi$. 
It follows that $x\in X_{{\bf m}, {\bf n}, \xi}^{\alpha,\beta,F}\subseteq X_{{\bf m}, {\bf n}, s, <t}^{\alpha,\beta,F}$. 

\begin{case*} 
$x_u\in E_{\bf m}^\alpha$. 
\end{case*} 
Note that the assumption $D^\alpha_{\bf m}\cap Y^{\alpha,\beta}=F^{-1}[D^\beta_{\bf n}]\cap Y^{\alpha,\beta}$ immediately implies that $E^\alpha_{\bf m}\cap Y^{\alpha,\beta}=F^{-1}[E^\beta_{\bf n}]\cap Y^{\alpha,\beta}$. 
Since $i\in \mathrm{Even}_{\bf m}$, we have  
$x_u\in C_{(2i-1,2i+1)}^\alpha\cap E_{\bf m}^\alpha=C_{(2i-1,2i)}^\alpha$ by the case assumption. 
Therefore $x_u\in E^\alpha_{\bf m}\cap Y^{\alpha,\beta}=F^{-1}[E^\beta_{\bf n}]\cap Y^{\alpha,\beta}$ and hence $F(x_u)\in E^\beta_{\bf n}$. 
Since moreover 
$F(x_u)\in F[U]\subseteq C_{(2j-1,2j+1)}^\beta$ and 
$j\in\mathrm{Even}_{\bf n}$, we have  
$$F(x_u)\in C_{(2j-1,2j+1)}^\beta\cap E_{\bf n}^\beta=C_{(2j-1,2j)}^\beta.$$ 
Since $x_u\in X^{{\bf m}, {\bf n},\alpha,\beta}_{\xi,F}$ and we just argued that $x_u\in C_{(2i-1,2i)}^\alpha$ and $F(x_u)\in C_{(2j-1,2j)}^\beta$, we have $(2i-1,2i)\sim_\xi (2j-1,2j)$ by the definition of $\sim_\xi$. 
We argued above that $c_{\bf m}(k)=c_{\bf m}(2i)=c_{\bf n}(2j)=c_{\bf m}(l)$. 
Therefore, the second case of Lemma \ref{extend unfolding} yields an unfolding $\xi^*$ 
that extends $\xi$ with $(2i,2i+1)\sim_{\xi^*}(2j,2j+1)$ and $\ran{\xi^*}=\ran{\xi}\cup\{k\}$. 
Since we argued in the beginning of the proof that $t>k$, it thus follows that $t\notin \ran{\xi^*}$ and $x\in X_{{\bf m}, {\bf n}, \xi^*}^{\alpha,\beta,F}\subseteq X_{{\bf m}, {\bf n}, s, <t}^{\alpha,\beta,F}$. 
\end{proof}

The previous claim completes the proof of Lemma \ref{Xi is closed} in the case above. We now discuss the necessary changes to the previous argument in the remaining cases. 

We first assume that $x\in E_{\bf m}^\alpha$. If $x_u\in D_{\bf m}^\alpha$, then $x_u\in C_{(2i,2i+1)}^\alpha$, $F(x_u)\in C_{(2j,2j+1)}^\beta$ and $x_u\in X^{{\bf m}, {\bf n},\alpha,\beta}_{\xi,F}$. 
Therefore $(2i,2i+1)\sim_\xi (2j,2j+1)$ by the definition of $\sim_\xi$. 
Since $i\in\mathrm{Odd}_{\bf m}$ and $j\in\mathrm{Odd}_{\bf n}$, we have $c_{\bf m}(2i)=c_{\bf n}(2j)$. 
By the first case of Lemma \ref{extend unfolding}, there is an unfolding $\xi^*=(f^*,g^*)$ 
that extends $\xi$ with $(2i-1,2i)\sim_{\xi^*}(2j-1,2j)$ and $\ran{f^*}=\ran{f}\cup\{k\}$. 
It thus follows that $x\in X_{{\bf m}, {\bf n}, \xi^*}^{\alpha,\beta,F}\subseteq X_{{\bf m}, {\bf n}, s, <t}^{\alpha,\beta,F}$. 
Otherwise, we have $x_u\in E_{\bf m}^\alpha$. It follows that $x_u\in C_{(2i-1,2i]}^\alpha$ and $F(x_u)\in C_{(2j-1,2j]}^\beta$. Since $x_u\in X_{{\bf m}, {\bf n}, \xi}^{\alpha,\beta,F}$, we have $(2i-1,2i)\sim_\xi (2j-1,2j)$ by the definition of $\sim_\xi$ and hence $x\in X_{{\bf m}, {\bf n}, \xi}^{\alpha,\beta,F}$.

Second, we assume that $k=2i+1$ and $l=2j+1$. Moreover, we first assume that $x\in D_{\bf m}^\alpha$. Then $i\in\Odd_{\bf m}$, $j\in \Odd_{\bf n}$ and hence $c_{\bf m}(2i+1)=c_{\bf n}(2j+1)$. 
If $x_u\in D_{\bf m}^\alpha$, then $x_u\in C_{(2i,2i+1]}^\alpha$, $F(x_u)\in C_{(2j,2j+1]}^\beta$ and hence $(2i,2i+1)\sim_\xi (2j,2j+1)$. 
It follows that $x\in X_{{\bf m}, {\bf n}, \xi}^{\alpha,\beta,F}$. 
If $x_u\in E_{\bf m}^\alpha$, then $x_u\in C_{(2i+1,2i+2)}^\alpha$, $F(x_u)\in C_{(2j+1,2j+2)}^\beta$ and hence $(2i+1,2i+2)\sim_\xi (2j+1,2j+2)$. 
By the first case of Lemma \ref{extend unfolding}, we obtain some $\xi^*$ with $(2i,2i+1)\sim_{\xi^*}(2j,2j+1)$ and $\ran{f^*}=\ran{f}\cup\{k\}$. Therefore $x\in X_{{\bf m}, {\bf n}, \xi^*}^{\alpha,\beta,F}\subseteq X_{{\bf m}, {\bf n}, s, <t}^{\alpha,\beta,F}$.

If $x\in E_{\bf m}^\alpha$, then $i\in\Even_{\bf m}$, $j\in \Even_{\bf n}$ and hence $c_{\bf m}(2i+1)=c_{\bf n}(2j+1)$. 
If $x_u\in D_{\bf m}^\alpha$, then $x_u\in C_{(2i,2i+1)}^\alpha$, $F(x_u)\in C_{(2j,2j+1)}^\beta$ and hence $(2i,2i+1)\sim_\xi (2j,2j+1)$. 
By the second case of Lemma \ref{extend unfolding}, we obtain some $\xi^*$ with $(2i+1,2i+2)\sim_{\xi^*}(2j+1,2j+2)$ and $\ran{f^*}=\ran{f}\cup\{k\}$, hence $x\in X_{{\bf m}, {\bf n}, \xi^*}^{\alpha,\beta,F}\subseteq X_{{\bf m}, {\bf n}, s, <t}^{\alpha,\beta,F}$. 
If $x_u\in E_{\bf m}^\alpha$, then $x_u\in C_{[2i+1,2i+2)}^\alpha$, $F(x_u)\in C_{[2j+1,2j+2)}^\beta$ and therefore $(2i+1,2i+2)\sim_\xi (2j+1,2j+2)$ 
and $x\in X_{{\bf m}, {\bf n}, \xi}^{\alpha,\beta,F}$.

Third, we assume that $k=2i$ and $l=2j+1$. Moreover, we first assume that $x\in D_{\bf m}^\alpha$. Then $i\in\Even_{\bf m}$, $j\in \Odd_{\bf n}$ and hence $c_{\bf m}(2i)=c_{\bf n}(2j+1)$. 
If $x_u\in D_{\bf m}^\alpha$, then $x_u\in C_{[2i,2i+1)}^\alpha$ and $F(x_u)\in C_{(2j,2j+1]}^\beta$. Therefore $(2i,2i+1)\sim_\xi (2j,2j+1)$ and $x\in X_{{\bf m}, {\bf n}, \xi}^{\alpha,\beta,F}$. 
If $x_u\in E_{\bf m}^\alpha$, then $x_u\in C_{(2i-1,2i)}^\alpha$, $F(x_u)\in C_{(2j+1,2j+2)}^\beta$ and hence $(2i-1,2i)\sim_\xi (2j+1,2j+2)$. 
By the last case of Lemma \ref{extend unfolding}, we obtain some $\xi^*$ with $(2i,2i+1)\sim_{\xi^*}(2j,2j+1)$ and $\ran{f^*}=\ran{f}\cup\{k\}$. Thus $x\in X_{{\bf m}, {\bf n}, \xi^*}^{\alpha,\beta,F}\subseteq X_{{\bf m}, {\bf n}, s, <t}^{\alpha,\beta,F}$.

If $x\in E_{\bf m}^\alpha$, then $i\in\Odd_{\bf m}$, $j\in \Even_{\bf n}$ and hence $c_{\bf m}(2i)=c_{\bf n}(2j+1)$. 
If $x_u\in D_{\bf m}^\alpha$, then $x_u\in C_{(2i,2i+1)}^\alpha$, $F(x_u)\in C_{(2j,2j+1)}^\beta$ and hence $(2i,2i+1)\sim_\xi (2j,2j+1)$. 
By the third case of Lemma \ref{extend unfolding}, we obtain some $\xi^*$ with $(2i-1,2i)\sim_{\xi^*}(2j+1,2j+2)$ and $\ran{f^*}=\ran{f}\cup\{k\}$. Hence $x\in X_{{\bf m}, {\bf n}, \xi^*}^{\alpha,\beta,F}\subseteq X_{{\bf m}, {\bf n}, s, <t}^{\alpha,\beta,F}$. 
If $x_u\in E_{\bf m}^\alpha$, then $x_u\in C_{(2i-1,2i]}^\alpha$ and $F(x_u)\in C_{[2j+1,2j+2)}^\beta$. Therefore $(2i-1,2i)\sim_\xi (2j+1,2j+2)$ 
and $x\in X_{{\bf m}, {\bf n}, \xi}^{\alpha,\beta,F}$.

Finally, we assume that $k=2i+1$ and $l=2j$. Moreover, we first assume that $x\in D_{\bf m}^\alpha$. Then $i\in\Odd_{\bf m}$, $j\in \Even_{\bf n}$ and hence $c_{\bf m}(2i+1)=c_{\bf n}(2j)$. 
If $x_u\in D_{\bf m}^\alpha$, then $x_u\in C_{(2i,2i+1]}^\alpha$ and $F(x_u)\in C_{[2j,2j+1)}^\beta$. Therefore $(2i,2i+1)\sim_\xi (2j,2j+1)$ and $x\in X_{{\bf m}, {\bf n}, \xi}^{\alpha,\beta,F}$. 
If $x_u\in E_{\bf m}^\alpha$, then $x_s\in C_{(2i+1,2i+2)}^\alpha$, $F(x_u)\in C_{(2j-1,2j)}^\beta$ and hence $(2i+1,2i+2)\sim_\xi (2j-1,2j)$. 
By the third case of Lemma \ref{extend unfolding}, we obtain some $\xi^*$ with $(2i,2i+1)\sim_{\xi^*}(2j,2j+1)$  and $\ran{f^*}=\ran{f}\cup\{k\}$. Thus  $x\in X_{{\bf m}, {\bf n}, \xi^*}^{\alpha,\beta,F}\subseteq X_{{\bf m}, {\bf n}, s, <t}^{\alpha,\beta,F}$.

If $x\in E_{\bf m}^\alpha$, then $i\in\Even_{\bf m}$, $j\in \Odd_{\bf n}$ and hence $c_{\bf m}(2i+1)=c_{\bf n}(2j)$. 
If $x_u\in D_{\bf m}^\alpha$, then $x_u\in C_{(2i,2i+1)}^\alpha$, $F(x_s)\in C_{(2j,2j+1)}^\beta$ and hence $(2i,2i+1)\sim_\xi (2j,2j+1)$. 
By the last case of Lemma \ref{extend unfolding}, we obtain some $\xi^*$ with $(2i+1,2i+2)\sim_{\xi^*}(2j,2j+1)$  and $\ran{f^*}=\ran{f}\cup\{k\}$. Hence $x\in X_{{\bf m}, {\bf n}, \xi^*}^{\alpha,\beta,F}\subseteq X_{{\bf m}, {\bf n}, s, <t}^{\alpha,\beta,F}$. 
If $x_u\in E_{\bf m}^\alpha$, then $x_u\in C_{[2i+1,2i+2)}^\alpha$ and $F(x_u)\in C_{(2j-1,2j]}^\beta$. Therefore $(2i+1,2i+2)\sim_\xi (2j-1,2j)$ 
and $x\in X_{{\bf m}, {\bf n}, \xi}^{\alpha,\beta,F}$. 
\end{proof}


\section{Incomparable non-definable sets}  \label{section: incomparable sets of arbitrary complexity}

The main result suggests the question whether it is possible to construct larger families of incomparable subsets of metric spaces of positive dimension. 
The next result shows that this is possible, if we make an additional assumption. 

\begin{theorem} \label{incomparable non-definable sets again}
Suppose that $(X,d)$ is a locally compact metric space of positive dimension. 
Then there is a (definable) injective map that takes sets of reals to subsets of $X$ in such a way that these subsets 
are pairwise incomparable. 
\end{theorem} 

\begin{proof} 
We will write 
$\cl(Y)$ for the closure 
of a subset $Y$ of $X$. 
Let $A$ denote the set of non-positive and $B$ the set of positive elements of $X$. 
Since $(X,d)$ has positive dimension, $B$ is nonempty. Let $\gamma\in B$. 
Since $(X,d)$ is moreover locally compact, there is an open ball $Y$ containing $\gamma$ that is pre-compact, i.e. its closure $\cl(Y)$ is compact. In particular, $Y$ is an open subset of a compact metric space and hence itself a Polish space.

It follows from the definition of $A$ that its dimension is $0$. 
If $B\cap Y$ is countable, then $Y$ has dimension $0$ by \cite[Theorem II.2]{MR0006493} 
and this contradicts the choice of $\gamma$. 
Therefore $B\cap Y$ is uncountable. 
Since  
$Y$ is pre-compact, it  
follows from the perfect set property for closed sets \cite[Theorem 6.2]{Kechris_Classical} that there is a perfect subset $C$ of $\cl(B)\cap Y$ 
that is nowhere dense in $\cl(B)\cap Y$. 

\begin{claim*}
There is a sequence ${\bm \alpha}=\langle \alpha_n\mid n\geq 0\rangle$ of distinct elements of $B\cap Y$ and a sequence ${\bf r}=\langle r_n\mid n\geq 0\rangle$ of positive real numbers converging to $0$ with the following properties for all distinct $i,j \in \NN$. 
\begin{enumerate-(a)} 
\item \label{condition a in last theorem} 
$\cl(B_{r_i}(\alpha_i))\subseteq Y\setminus C$. 
\item \label{condition b in last theorem} 
$B_{r_i}(\alpha_i)\cap B_{r_j}(\alpha_j)=\emptyset$. 
\item \label{condition c in last theorem} 
$C\subseteq \cl(\{\alpha_n\mid n\geq 0\})$. 
\end{enumerate-(a)} 
\end{claim*} 

Since the construction of such sequences is immediate, we omit the proof. 
Note that these conditions remain true if ${\bm \alpha}$ is fixed and each $r_n$ is decreased. 
Recall that $r^{\alpha }>0$ was chosen in the beginning of Section \ref{section: incomparable Borel sets} such that $\alpha$ is $B_{r^\alpha}(\alpha)$-positive for any positive $\alpha$ and $X^\alpha=B_{r^\alpha}(\alpha)$. 
Again, this remains true if $\alpha$ is fixed and $r^{\alpha }$ is decreased. 
Therefore, we can assume that $r_n=r^{\alpha_n}$ for all $n\geq 0$.



Moreover, we fix a sequence $\langle {\bf m}_n\mid n\ge 0 \rangle$ of strictly increasing 
sequences ${\bf m}_n$ in $\NN^\NN$ beginning with $0$ such that $\Delta{\bf m}_n$ is strictly increasing 
and moreover, the sequences $\Delta{\bf m}_n$ are pairwise not $\mathrm{E}_{\mathrm{tail}}$-equivalent for different $n\geq 0$. 
We now consider the subsets $D_{{\bf m}_n}^{\alpha_n}$ of $X^{\alpha_n}$ that were given in Definition \ref{definition: blocks indexed by sequences}. 
We further fix a bijection $h\colon 2^{\NN}\rightarrow C$ and for each subset $I$ of $2^\NN$, let 
$$D_I=h[I]\cup\bigcup_{n\geq 0}D_{{\bf m}_n}^{\alpha_n}.$$


For the following claims, we fix some $n\geq 0$. Let $A^*$ denote the set of non-$X^{\alpha_n}$-positive and $B^*$ the set of $X^{\alpha_n}$-positive elements of $X^{\alpha_n}$. 
Then $A^*$ is an open and $B^*$ a relatively closed subset of $X^{\alpha_n}$ and moreover $\alpha_n\in B^*$. 




\begin{claim*} 
Suppose that $0<s<r_n$ and $U$ is a subset of $B^*\cap B_s(\alpha_n)$ with $\alpha_n\in U$ that is open and closed in $B^*\cap B_s(\alpha_n)$. 
Then for any $r<s$, there is some $x\in U$ with $d(\alpha_n,x)=r$. 
\end{claim*} 
\begin{proof} 
The claim holds for $r=0$, since $\alpha_n\in U$. 
Assuming that the claim fails for some $r$ with $0<r<s$, let 
$\bar{B}_r(\alpha_n)=\{x\in X\mid d(\alpha_n,x)\leq r\}$ denote the closed ball of radius $r$ around $x_n$ in $X$ and 
$V=U\cap B_r(\alpha_n)=U\cap \bar{B}_r(\alpha_n)$. 
Then $V$ is a subset of $B^*$ with $\alpha_n\in V$ that is open and closed in $B^*$ and the values $d(\alpha_n,x)$ for $x\in V$ are bounded strictly below $s$. 
Thus it is sufficient to show that there is no such set. 

We choose some $\epsilon>0$ with $d(\alpha_n,x)<s-\epsilon$ for all $x\in V$. 
We further claim that there is some $\delta$ with $0<\delta\leq\epsilon$ and $d(x,y)>\delta$ 
for all $(x,y)\in V\times (B^*\setminus V)$. 
Assuming that this fails, it follows from the fact that $Y$ is pre-compact that there exist two sequences in $V$ and $B^*\setminus V$ that both converge to the same $y\in \cl(Y)$. Since $B^*$ is closed in $X^{\alpha_n}$ and $d(\alpha_n,y)<r_n$, we have $y\in B^*$. 
But this contradicts the fact that $V$ is open and closed in $B^*$. 


We fix some such $\delta$ and let $S_{\sim}=\{x\in B_{s}(\alpha_n)\mid d(x,V)\sim \delta\}$ for $\sim\in \{=,\leq,<\}$. 
Since $S_=$, $S_\leq$ are closed subsets of $B_s(\alpha_n)$ and $d(\alpha_n,x)\leq s-\epsilon$ for all $x\in S_\leq$ by the choice of $\epsilon$ and $\delta$, these sets are closed in $\cl(Y)$ and in $X$. 

By the definition of $A^*$, 
there is a family $\langle A_k\mid k\in K\rangle$ of sets 
that are both open and closed in $X$ with $A^*=\bigcup_{k\in K} A_k$. 
Since $S_=$ is a closed subset of the compact set $\cl(Y)$, it is itself compact. 
Since $S_=$ is disjoint from $B^*$ by the choice of $\delta$ and is thus a subset of $A^*$, there is a finite set $L\subseteq K$ with $S_=\subseteq \bigcup_{l\in L} A_l$. 
Let $W=S_{\leq}\setminus \bigcup_{l\in L} A_l=S_{<}\setminus \bigcup_{l\in L} A_l$. 
Since the sets $A_l$ are open and closed, $S_\leq$ is closed and $S_<$ is open in $X$, we have that $W$ is a subset of $B_{s}(\alpha_n)$ that contains $\alpha_n$ and is open and closed in $X$. 
However, this contradicts the fact that $\alpha_n\in B^*$. 
\end{proof}

It remains to show that $D_I$, $D_J$ are incomparable whenever  $I$, $J$ are distinct subsets of $2^\NN$. 
Towards a contradiction, assume that $F\colon X\rightarrow X$ is a continuous map with $D_I=F^{-1}[D_J]$. 
Let $B_*=B^*\cap B_{r_1^{\alpha_n}}(\alpha_n)$.

\begin{claim*} 
There is some $x\in B_*$ with $F(x)\notin C$. 
\end{claim*} 
\begin{proof} 
Towards a contradiction, we assume that $\ran{F{\upharpoonright}B_*}\subseteq C$. 
It follows that for any subset $V$ of $C$ with $F(\alpha_n)\in V$ that is open and closed in $C$, its preimage $U=(F{\upharpoonright}B_*)^{-1}[V]$ contains $\alpha_n$ and is 
open and closed in $B_*$. 
By the previous claim applied to $s=r_1^{\alpha_n}$, 
for any $r<r_1^{\alpha_n}$ there is some $x\in U$ with $d(\alpha_n,x)=r$. 

We now apply this statement to arbitrarily small neighborhoods $V$ of $F(\alpha_n)$ in $C$ that are open and closed in $C$ and to arbitrarily large $r<r_1^{\alpha}$ to obtain a sequence ${\bf x}=\langle x_i\mid i\geq 0\rangle$ of elements of $U$ as above. 
Since $Y$ is pre-compact,  
${\bf x}$ has a subsequence converging to some $x\in \cl(Y)$ with $d(\alpha_n,x)=r_1^{\alpha_n}$ and $F(x)=F(\alpha_n)$.

Since ${\bf m}_n$ begins with $0$, we have $0\in \Even_{{\bf m}_n}$ by Definition \ref{definition: even and odd blocks} and thus the first block of $D_{{\bf m}_n}^{\alpha_n}$ is $C_{[0,1)}^{\alpha_n}=B_{r_1^{\alpha_n}}(\alpha_n)$ by Definition \ref{definition: blocks indexed by sequences}. 
Since $d(\alpha_n,x)=r_1^{\alpha_n}$, this implies that $x\notin D_{{\bf m}_n}^{\alpha_n}$ and $x\notin D_I$ by the definition of $D_I$ before the second claim and the conditions \ref{condition a in last theorem} and \ref{condition b in last theorem} in the first claim. 
Since moreover $\alpha_n\in D_{{\bf m}_n}^{\alpha_n}\subseteq D_I$ and $F(x)=F(\alpha_n)$ by the choice of $x$, this contradicts the assumption that $F$ is a reduction of $D_I$ to $D_J$. 
\end{proof}

We now fix some $x\in B_*$ as in the last claim. 
Since $x\in B_*=B^*\cap B_{r_1^{\alpha_n}}\subseteq B_{r_1^{\alpha_n}} \subseteq D_{{\bf m}_n}^{\alpha_n}\subseteq D_I= F^{-1}[D_J]$, we have $F(x)\in D_J$. 
Since moreover $F(x)\notin C$ by the choice of $x$, there is some $k\in\NN$ with $F(x)\in D^{\alpha_k}_{{\bf m}_k}$ by the definition of $D_J$.


\begin{claim*} 
$\Delta {\bf m}_n$, $\Delta {\bf m}_k$ are $\mathrm{E}_{\mathrm{tail}}$-equivalent. 
\end{claim*} 
\begin{proof} 
Let $Y^{\alpha_n,\alpha_k}=X^{\alpha_n}\cap F^{-1}[X^{\alpha_k}]$ as given before Definition \ref{definition compatible}. 
We have $D_{{\bf m}_n}^{\alpha_n}\cap X^{\alpha_n}=D_I\cap X^{\alpha_n}$ and $D_{{\bf m}_k}^{\alpha_k}\cap X^{\alpha_k}=D_J\cap X^{\alpha_k}$ by the definition of $D_I$, $D_J$ and condition \ref{condition a in last theorem} in the first claim. 
Moreover, we made the assumption that $D_I=F^{-1}[D_J]$. 
Hence $D_{{\bf m}_n}^{\alpha_n}\cap Y^{\alpha_n,\alpha_k}= D_I\cap Y^{\alpha_n,\alpha_k}=F^{-1}[D_J] \cap Y^{\alpha_n,\alpha_k}=F^{-1}[D_J\cap X^{\alpha_k}] \cap Y^{\alpha_n,\alpha_k}=F^{-1}[D_{{\bf m}_n}^{\alpha_n}\cap X^{\alpha_k}] \cap Y^{\alpha_n,\alpha_k}=F^{-1}[D_{{\bf m}_n}^{\alpha_n}] \cap Y^{\alpha_n,\alpha_k}$. 
In particular, the conditions stated before Definition \ref{definition compatible} hold for ${\bf m}_n$, ${\bf m}_k$, $\alpha_n$, $\alpha_k$, $F$. 
We further have $x\in B_{r_1^{\alpha_n}}(\alpha_n)$, $F(x)\in D^{\alpha_k}_{{\bf m}_k}\subseteq X^{\alpha_k}$ and $x\in B^*$ by the choice of $x$, $k$ before this claim. These facts show that Lemma \ref{existence of reduction implies tail equivalence} can be applied to the parameters $\alpha=\alpha_n$, $\beta=\alpha_k$, ${\bf m}={\bf m}_n$, ${\bf n}={\bf m}_k$, $\gamma=x$ and $F$. 
The claim thus follows from Lemmas \ref{existence of reduction implies tail equivalence} and \ref{unfoldings with large domain implies tail equivalence}. 
\end{proof}


The last claim implies that $n=k$ by the definition of $\Delta {\bf m}_n$, $\Delta {\bf m}_k$. 
Thus the last two claims show that for every $n\geq 0$, there is some $x\in B_{r_1^{\alpha_n}}(\alpha_n)\subseteq B_{r_{n}}(\alpha_n)$ with $F(x)\in D^{\alpha_n}_{{\bf m}_n} \subseteq B_{r_{n}}(\alpha_n)$. 
Since $C\subseteq \cl(\{\alpha_n\mid n\geq 0\})$ by condition \ref{condition c in last theorem} in the first claim and ${\bf r}=\langle r_n\mid n\geq 0\rangle$ converges to $0$, it follows that $F{\upharpoonright} C=\mathrm{id}{\upharpoonright} C$. 
We further have $D_I\cap C=h[I]$ and $D_J\cap C=h[J]$ by the definition of $D_I$, $D_J$. 
Since we assumed that $D_I=F^{-1} [D_J]$, it follows that $h[I]=D_I\cap C=F^{-1} [D_J]\cap C= D_J\cap C=h[J]$. 
Since  $h\colon 2^{\NN}\rightarrow C$ is a bijection, this implies that $I=J$, but this contradicts our assumption that $I$, $J$ are distinct. 
\end{proof}


\section{Further remarks} \label{section: further remarks}

We state a few further observations about the main theorem.

\begin{remark} 
It follows from Theorem \ref{main theorem} that there are totally disconnected Polish spaces with uncountably many incomparable Borel subsets, for instance the complete Erd\"os space (see \cite{MR2488452}). 
\end{remark}


By Urysohn's metrization theorem, the conclusion of the main theorem holds for countably based regular Hausdorff spaces, but fails without this requirement 
by the next remark. 

\begin{remark} 
The conclusion of Theorem \ref{main theorem} fails for countable $T_0$ spaces by \cite[Remark 5.35]{MR3417077}. 
\end{remark}

Moreover, the conclusion of the main theorem 
is optimal in the sense that the next remark prevents further embedding results, unless additional properties of the space are assumed. 

\begin{remark} 
There is a compact connected subspace of $X$ of $\mathbb{R}^3$ such that any two nonempty subsets 
that are 
not equal to $X$ are incomparable 
\cite[Theorem 11]{MR0220249} (see the remark after Theorem 5.15 in \cite{MR3417077}). 
\end{remark} 

We finally remark that the construction in the proof of the main theorem can also be used to prove other embedding results.

\begin{remark} 
If there is a partition of a metric space $(X,d)$ into infinitely many subspaces of positive dimension that are both open and closed, then $(\mathcal{P}(\NN),\subseteq)$ embeds into the Wadge quasi-order on the collection of Borel subsets of $(X,d)$. This can be proved by applying the construction in the proof of the main theorem to the subspaces.  
\end{remark}

\section{Questions}

We conclude with some open questions. 
Since the sets $D_{\bf n}$ defined in the proof of the main theorem 
are intersections of open and closed sets, this suggests the following question. 

\begin{question} 
Does the conclusion of Theorem \ref{main theorem} 
hold for sets $A$ such that both $A$ and its complement is an intersection of an open and a closed set?  
\end{question} 

It is further open whether it is necessary for the proof of the main theorem to assume that the space is a metric space. 

\begin{question} 
Does the conclusion of Theorem \ref{main theorem} hold for all 
regular spaces of positive dimension? 
\end{question} 

Moreover, it is open whether local compactness can be omitted 
in the construction of incomparable non-definable sets in Section \ref{section: incomparable sets of arbitrary complexity}. 

\begin{question} 
Does the conclusion  of Theorem \ref{incomparable non-definable sets} hold for all Polish spaces of positive dimension? 
\end{question} 


In a different direction, it would be interesting to consider similar problems for functions on arbitrary metric spaces 
(see e.g. \cite{MR3145199, MR1887681}).

\bibliographystyle{alpha}
\bibliography{references}

\end{document}